\documentclass[a4paper,reqno,12pt]{amsart}
\usepackage{amsmath}
\usepackage{amssymb}
\usepackage{mathdots}
\usepackage{tcolorbox}

\usepackage{amsthm}
\usepackage{tikz,amsmath}
\usetikzlibrary{arrows}
\usetikzlibrary{calc}

\usepackage{tikz-cd}

\usetikzlibrary{decorations.pathreplacing}
\tikzset{%
  show curve controls/.style={
    postaction={
      decoration={
        show path construction,
        curveto code={
          \draw [blue,-]
            (\tikzinputsegmentfirst) -- (\tikzinputsegmentsupporta)
            (\tikzinputsegmentlast) -- (\tikzinputsegmentsupportb);
          \fill [red, opacity=0.5]
            (\tikzinputsegmentsupporta) circle [radius=.2ex]
            (\tikzinputsegmentsupportb) circle [radius=.2ex];
        }
      },
      decorate
    }
  },
  scc/.style={
  },
}
  \usepackage{caption}
\usepackage{epsfig}
\usepackage{graphicx}
\usepackage[normalem]{ulem}
\usepackage[all,line,
]{xy}
\CompileMatrices %goer at xy-matricerne koerer hurtigere.
\newcommand{\id}{\operatorname{id}}

 \newcommand{\Rec}{\operatorname{Rec}}

\newcommand{\N}{{\mathbb{N}}}

\newcommand{\rec}{\operatorname{rec}}

\newcommand{\p}{\operatorname{per}}

\DeclareMathOperator{\per}{Per}

   \theoremstyle{plain}%default
   \newtheorem{thm}{Theorem}[section]
   \newtheorem{prop}[thm]{Proposition}
   \newtheorem{lemma}[thm]{Lemma}
   \newtheorem{cor}[thm]{Corollary}
   \theoremstyle{definition}
   
   \newtheorem{defn}[thm]{Definition}
   \newtheorem{example}[thm]{Example}
   \theoremstyle{remark}
   
   \newtheorem{remark}[thm]{Remark}

\newtheorem{?}[thm]{Question}

\usepackage{mdframed,xcolor}

\definecolor{mybgcolor}{gray}{0.8}
\definecolor{myframecolor}{rgb}{.647,.129,.149}

\mdfdefinestyle{mystyle}{
  usetwoside=false,
  skipabove=0.6em plus 0.8em minus 0.2em,
  skipbelow=0.6em plus 0.8em minus 0.2em,
  innerleftmargin=.25em,
  innerrightmargin=0.25em,
  innertopmargin=0.25em,
  innerbottommargin=0.25em,
  leftmargin=-.75em,
  rightmargin=-0em,
  topline=false,
  rightline=false,
  bottomline=false,
  leftline=false,
  backgroundcolor=mybgcolor,
  splittopskip=0.75em,
  splitbottomskip=0.25em,
  innerleftmargin=0.5em,
  leftline=true,
  linecolor=myframecolor,
  linewidth=0.25em,
}

\newmdenv[style=mystyle]{important}

\usepackage{lipsum}

   \numberwithin{equation}{section}

\title[Factorizable embeddings]{Factorizable embeddings and the period of an irreducible sofic shift}
\author[B. Marcus]{Brian Marcus}

\address{The University of British Columbia, Vancouver}
\email{marcus@math.ubc.ca}

\author[T. Meyerovitch]{Tom Meyerovitch}
 \address{Ben Gurion University of the Negev. Departement of Mathematics. Be’er Sheva, 8410501,
Israel}

\email{mtom@bgu.ac.il}

\author[K. Thomsen]{Klaus Thomsen}
 \address{Institut for matematiske fag, Ny Munkegade, 8000 Aarhus C, Denmark}

\email{matkt@math.au.dk}

\author[C. Wu]{Chengyu Wu}

\address{The University of British Columbia, Vancouver}
\email{ chengyuw@connect.hku.hk}

\begin{document}

%Flg. kommando kan justere hvor detaljeret Indholdsfortegnelse (Table Of Contents)
%bliver.
%\setcounter{tocdepth}{3}

\begin{abstract} Generalizing a result of MacDonald we give necessary and sufficient conditions for an arbitrary subshift to embed into an irreducible sofic shift factoring through a given cover by an irreducible subshift of finite type (SFT). We obtain also necessary and sufficient conditions for an arbitrary subshift to embed into an irreducible sofic shift factoring through \emph{some} sliding block code out of an irreducible SFT. We do that when the code is required to be surjective, and hence a factor code, and when it is required to be injective {or almost invertible}, or is allowed to be arbitrary. These results require  concepts of the period of an irreducible sofic shift as well as a concept of a $p$-periodic subshift. Several equivalent formulations of the period are developed.
\end{abstract}

\maketitle

%\tableofcontents

%----------Selve teksten -------------------------------------------------------

\section{Introduction}

A celebrated theorem of Wolfgang Krieger, \cite{Kr1}, gives necessary and sufficient conditions for a subshift to embed into a given mixing SFT of larger entropy. To formulate it, let $Q_n(X)$ denote the set of periodic points of least period $n$ in a subshift $X$ and let $q_n(X):= \# Q_n(X)$ be the number of elements in $Q_n(X)$. Then Krieger's theorem reads as follows:

\begin{thm}\label{Krieger} (W. Krieger) Let $Z$ be a subshift and $Y$ a mixing SFT such that $h(Z) < h(Y)$. Then $Z$ embeds into $Y$ if and only if $q_n(Z) \leq q_n(Y)$ for all $n \in \mathbb N$.
\end{thm}

Since a proper subshift of an irreducible SFT must have strictly smaller entropy, Krieger's theorem gives in fact both necessary and sufficient conditions for a subshift to embed into a mixing SFT. It is striking that the rather obvious necessary condition on the periodic points is also sufficient when $h(Z) < h(Y)$.

Despite much effort and several partial results, no one has been able to give necessary and sufficient conditions for a subshift to embed into an arbitrary given sofic shift, be it irreducible or mixing. The lesson seems to be that although sofic shifts may appear to be only slight generalizations of SFTs, they contain a much more rigid internal structure; a structure which is still not fully understood. Nonetheless, a stronger version of embedding can be considered:
\begin{defn}\label{factorthrough}
	Let $X,Y$ and $Z$ be subshifts and $\pi: X\to Y$ a sliding block code. Let $\phi: Z\to Y$ be an embedding. We say that $\phi$ {\em factors through $\pi$} if there is a sliding block code $\psi : Z \to X$ making the following diagram commute:
\end{defn}
\begin{equation}\label{20-03-23cx}
	\xymatrix{
		& X \ar[d]^-{\pi}  \\
		Z \ar[ur]^-\psi \ar[r]_-\phi & Y}
\end{equation}
 Note that $\psi$ is automatically an embedding. 
 When $\pi$ is surjective, $X$ is an irreducible SFT and $Y$ is, consequently, an irreducible sofic shift, we refer to $\pi: X \to Y$ as a \emph{cover} of $Y$.

In her thesis Sophie MacDonald considered the following problem: Given a cover $\pi : X \to Y$ of $Y$, where $X$, and hence also $Y$, is mixing, when can a subshift $Z$ with $h(Z) < h(Y)$ be embedded into $Y$ in such a way that the embedding factors through $\pi$?
%See \cite{M} for the published version of her work on this.
Given the difficulty others have encountered when trying to embed into an arbitrary sofic shift, and despite that the problem she considered seems to be more complicated because it is asking for a specific type of embedding, MacDonald was able to give the following simple answer: Let $r_n(\pi) := \# Q_n(Y) \cap \pi(Q_n(X))$. Then the desired embedding exists if and only if $q_n(Z) \leq r_n(\pi)$ for all $n \in \mathbb N$.

\begin{thm}\label{Sophie} (MacDonald's Theorem \cite{M}) Let $\pi : X \to Y$ be a factor code from a mixing SFT $X$ to a sofic shift $Y$. Let $Z$ be a subshift such that $h(Z) < h(Y)$. Then, there is an embedding from $Z$ into $Y$ which factors through $\pi$ if and only if $q_{n}(Z) \leq r_{n}(\pi)$ for all $n \in \mathbb N$.
\end{thm}

 As with Krieger's theorem, an obvious necessary condition concerning periodic points turns out to be sufficient for the embedding in Theorem \ref{Sophie}.  The condition is closely related to a condition involving the concept of a receptive periodic point (see Definition~\ref{receptive_def}), which Boyle~\cite{B} used as a sufficient condition for an ordinary embedding of a subshift into a mixing sofic shift.  In effect, MacDonald's result gives a related necessary and sufficient condition for a stronger notion of embedding, i.e., an embedding which  factors  through a given factor code on a mixing SFT $X$.

The first objective of the present paper is to extend MacDonald's Theorem to the case where $X$ is merely an irreducible SFT and to give both necessary and sufficient conditions also in that case. This is achieved in Theorem \ref{Sophie2}, along with a second version, Theorem \ref{Sophie3} where there is no assumption on the entropies.  The proofs of these results are similar in spirit, but  somewhat simpler, than MacDonald's proof. One virtue of our approach, besides its greater generality, is that it allows us to use the statement of Krieger's theorem rather than modifying its proof. Along the way, in Proposition~\ref{prop:irreducible_emb}, we generalize also Krieger's original theorem to the case where the target shift space is an irreducible, but not necessarily mixing, SFT. For this the notion of $p$-period for a general subshift is introduced as a necessary part of the assumptions. See Definition \ref{p-periodic}.

MacDonald's Theorem was, in part, motivated by the following setup.
In information theory, a communication channel is modeled using a set of transition probabilities, with the input and output described by stationary probability measures. The channel's noise arises from the randomness of its transition probabilities. When there is no noise, the channel is referred to as a distortion channel, which, from the perspective of symbolic dynamics, can, in an abstract way be interpreted as a factor code from a subshift of channel inputs to a subshift of channel outputs.
We can then view the problem of encoding messages through the channel as the problem of embedding a subshift of messages into the domain of the factor code such that the restriction of the factor code to the range of the embedding is one-to-one.

 The second aim of the paper is to find necessary and sufficient conditions for an arbitrary subshift $Z$ to embed into an irreducible sofic shift $Y$ in such a way that the embedding factors through {\em some} sliding block code $\pi: X \to Y$ out of an irreducible SFT $X$. We obtain such a characterization in Theorem \ref{18-06-24} and Theorem \ref{05-06-25e}, when there are no conditions on $\pi$, when it is required to be injective and finally also when it is required to be surjective.

 A crucial ingredient in Theorem \ref{18-06-24} and Theorem \ref{05-06-25e} is a new notion of period for an irreducible sofic shift. In Section~\ref{period}, we consider several additional notions of a period for an irreducible sofic shift and show that these are all equivalent. We also consider the concept of a $p$-periodic subshift, %introduced in Definition~\ref{p-periodic}, 
 as mentioned above, which applies to general subshifts.  We show that $p$-periodicity is rather different from the notion of period in the case where the subshift is irreducible sofic. 

%{\color{blue} In Section \ref{sec-AI-factorizable}, we consider embeddings of an arbitrary subshift $Z$ into an irreducible sofic shift $Y$ that factor through {\em some almost invertible} factor code $\pi: X \to Y$ on an irreducible SFT $X$. We show that the existence of such an embedding  is equivalent to the existence of an embedding which merely factors through a factor code on an irreducible SFT.}
%\comm{Klaus}{I suggest the following formulation here.}
%\comm{Chengyu}{Done.}
{In Section \ref{sec-AI-factorizable} we show that if there is an embedding of an arbitrary subshift $Z$ into an irreducible sofic shift $Y$ that factors through a cover of $Y$, then that cover can be arranged to be almost invertible.}

 We recently became aware of a related paper by Wolfgang Krieger \cite{Kr2}.

\section{Notation, terminology, and background}\label{notations}

%\begin{remark} \label{25-01-07-a}
%	The above lemma replaces $X$ and $\pi$ by another $1$-step SFT  $X_G$ and $1$-block code $\mathcal{L}$. This technique is commonly used in symbolic dynamics and often referred to as ``recode $\pi: X\to Y$ so that $X$ is $1$-step and $\pi$ is $1$-block."
%\end{remark}

%{\color{blue}{SHOULD WE DEFINE THESE THINGS HERE OR DEFINE THEM AS WE NEED THEM?

%Consistency:

%Minimal period vs least period if a periodic point?

%Minimal right (left) resolving presentation vs right (left) Fischer cover of an irreducible sofic shift?

%Irreducible vs strongly connected graph

%Does N include 0? (I was once told that N does not contain 0 because, according to the Catholic Church, 0 is not natural becauses it connotes the lack of chilren)

%Define:

%irreducible component of irreducible sofic shift

%top component

%$Y_c$

%derived shift

%Properties of right Fischer cover: follower separated, almost invertible

%Define magic

%Define follower separated

%}}

	We introduce in this brief section some notation, terminology and basic results from symbolic dynamics, and we refer the reader to \cite{LM} for further information. We do assume that the reader is familiar with the fundamental concepts of shift space (or subshift), the language ${\mathcal B}(X)$ of a shift space $X$, irreducible and mixing shift spaces, and sliding block codes, factor codes, embeddings and conjugacies from one shift space to another.
%	\comm{klaus}{I suggest to remove the following sentences, and place them where they are used; perhaps in a bracket or in a footnote.}
%	\comm{Chengyu}{The paragraph has been moved to section 6 as a footnote.}

A {\em shift of finite type (SFT)} is a shift space obtained by forbidding finitely many words from a full shift ${\mathcal A}^{\mathbb Z}$, where
$\mathcal A$ is a finite alphabet. For a finite directed graph $G$, we use $X_G$ to denote the edge shift of $G$, i.e., the SFT whose elements are all bi-infinite edge paths in $G$.
	
	A {\em labeled graph} $\mathcal G = (G, \mathcal{L})$ is a finite directed graph $G$, together with a label function $\mathcal{L}$, which maps edges of $G$ into some finite alphabet $\mathcal{A}$. The initial vertex of a path or edge $\gamma$ in $G$ will be denoted by $i_G(\gamma)$, and the terminal vertex of $\gamma$ by $t_G(\gamma)$.

	Let $\mathcal{L}_\infty$ be the $1$-block code which maps each bi-infinite (unlabeled) edge path in $G$ to its label sequence.
A {\em sofic shift} $Y$ is the set of all bi-infinite label sequences of a labeled graph $(G, \mathcal{L})$, i.e.,
$$
Y = X_{\mathcal G} = \{\mathcal{L}_\infty(\gamma): \gamma \mbox{ is a bi-infinite edge path in } G\}.
$$
Such a labeled graph $(G, \mathcal{L})$ is called a {\em presentation} of $Y$. 

{A sofic shift can alternatively be represented by a vertex labeled graph, for which we use the same notation $(G,\mathcal{L})$ where $\mathcal{L}$ is a label function on vertices of $G$.
%\comm{klaus}{I suggest to remove the following sentence:}
%\comm{Chengyu}{Done.}
%A sofic shift can equivalently be defined as a subshift that is a factor of an SFT.
}

	For each vertex $I$ of $G$, the follower set of $I$ is the collection of label sequences of finite edge paths starting at $I$. A presentation $(G, \mathcal L)$  is said to be {\em follower-separated} if distinct vertices have distinct follower sets. It is said to be  {\em right-resolving} (resp. {\em left-resolving}) if for all $a\in \mathcal{A}$, each vertex in $G$ has at most one outgoing ({\em resp.} incoming) edge labeled $a$.
And it is said to be irreducible if $G$ is irreducible. 	
	
	A {\em minimal right-resolving presentation} of a sofic shift $Y$ is a right-resolving presentation which has the fewest number of vertices among all right-resolving presentations of $Y$. When $Y$ is irreducible, its minimal right-resolving presentation is unique up to labeled graph isomorphism, and we call it the {\em right Fischer cover}. The left Fischer cover is defined similarly. In the remainder of the paper, the term Fischer cover is used to refer to the right Fischer cover unless stated otherwise.

 It is well known that the right Fischer cover $(G,\mathcal{L})$ of an irreducible sofic shift $Y=X_{\mathcal G}$ is characterized as the unique irreducible, right-resolving and follower-separated presentation of $Y$ (\cite[Corollary 3.3.19]{LM}). And the induced 1-block code $\mathcal{L}_\infty$ is {\em almost invertible,} i.e., every doubly transitive point has a unique preimage; here, a point $y\in Y$ is {\em doubly transitive} if every word in the language of $Y$ appears infinitely often to the left and to the right in $y$.

For an irreducible SFT $X$ the \emph{global period} $p$ is the greatest common divisor (gcd) of the periods realized by periodic points in $X$. There is then a partition $X = \sqcup_{i=0}^{p-1} X_i$ of $X$ into closed sets such that $\sigma(X_i) = X_{i+1}$ with addition mod $p$, such that $\sigma^p$ is mixing on the $X_i$'s. This partition is unique up to cyclic permutation and will be called \emph{the canonical cyclic partition} of $X$. We denote the global period of $X$ by $\text{per}(X)$.	

{Synchronizing elements of $\mathcal B(Y)$ will be of particular importance throughout the paper:
 A word $u \in \mathcal{B}(Y)$ is \emph{synchronizing} if for all words $a, b$ in $\mathcal{B}(Y)$,
$$
au, ub \in \mathcal{B}(Y) \ \Rightarrow \ aub \in \mathcal{B}(Y) .
$$
This can be viewed as an independence condition: the constraints on words imposed by the language of $Y$ on the left of $w$ are independent of those on the right.  A characterization of an SFT in terms of its language~\cite[Theorem 2.1.8]{LM} is precisely that of a shift space such that every sufficiently long allowed word is synchronizing. 
Note also that any word that contains a synchronizing word is synchronizing. }
	
	We end this section with two typical recoding results in symbolic dynamics.
	
	\begin{lemma}\label{08-05-24c} Let $Y$ be an irreducible sofic shift and $\pi : X \to Y$ a cover of $Y$. There is an irreducible graph $G$ and a labeling $\mathcal{L}$ of the edges of $G$ by the letters of the alphabet of $Y$ and a conjugacy $\psi : X \to X_G$ such that
		\begin{equation*}
			\begin{xymatrix}{
					X\ar[dr]_-{\pi}\ar[rr]^-\psi & &X_G\ar[dl]^-{\mathcal{L}_\infty} \\
					& Y & }
			\end{xymatrix}
		\end{equation*}
		commutes.
	\end{lemma}
	\begin{proof} This follows from (the proof of) \cite[Theorem 3.2.1]{LM}.
	\end{proof}

	\begin{lemma}\label{08-05-24d} Let $(G, \mathcal{L})$ be an irreducible labeled graph and $Y \subseteq X_G$ an irreducible sub-SFT of $X_G$. Then, there exist a finite irreducible labeled graph $(G',\mathcal{L}')$, an irreducible subgraph $H \subseteq G'$ and conjugacies $\psi : X_G \to X_{G'}$ and $\phi : Y \to X_H$ such that the following diagram commutes.
		\begin{equation}\label{08-05-24e}
			\begin{xymatrix}{
					& X_\mathcal G & \\
					X_G \ar[ur]^-{\mathcal{L}_\infty}\ar[rr]^-\psi & &X_{G'}\ar[ul]_-{\mathcal{L}'_\infty} \\
					Y \ar@{^{(}->}[u] \ar[rr]^-{\phi}  &  &X_H \ar@{^{(}->}[u]  }
			\end{xymatrix}
		\end{equation}
	\end{lemma}
	\begin{proof} Let $Y$ be $m$-step \cite[Definition  2.1.6]{LM}. The set of vertices of the graph $G'$ consists of the words $\left\{ x_{[-m,m]} : \ x \in X_G \right\}$ and the set of vertices of the subgraph $H$ consists of the words $\left\{ x_{[-m,m]} : \ x \in Y \right\}$. The set of edges of the graph $G'$ consists of the blocks $\left\{ x_{[-m,m+1]} : \ x \in X_G \right\}$, with
		$$
		i_{G'}(x_{[-m,m+1]}) = x_{[-m,m]}, \ t_{G'}(x_{[-m,m+1]}) = x_{[-(m-1),m+1]} .
		$$
		The edges of $H$ consist of the edges $x_{[-m,m+1]}$ from $G'$ for which $x \in Y$. The labeling $\mathcal L'$ is defined such that $\mathcal L'(x_{[-m,m+1]}) = \mathcal L(x_0)$ and $\psi : X_G \to X_{G'}$ is defined such that
		$$
		\psi\left((x_i)_{i \in \mathbb Z}\right) = \left( x_{[i-m,m-i+1]}\right)_{i \in \mathbb Z} .
		$$
		Finally, set $\phi:= \psi|_Y$.

	\end{proof}

\section{Encoding subshifts through an irreducible SFT - a generalization of Macdonald's theorem} \label{Sophie5}

 %We shall generally use terminology and notation as in the book \cite{LM}, and we will introduce any additional notation when it is needed.

%\begin{defn}\label{factorthrough}
%Let $X,Y$ and $Z$ be subshifts and $\pi: X\to Y$ a sliding block code. Let $\phi: Z\to Y$ be an embedding. We say that $\phi$ {\em factors through $\pi$} if there is a {\color{blue}{sliding block code}} $\psi : Z \to X$ making the following diagram commute:
%\end{defn}
%\begin{equation}\label{20-03-23cx}
%	\xymatrix{
%		& X \ar[d]^-{\pi}  \\
%		Z \ar[ur]^-\psi \ar[r]_-\phi & Y}.
%\end{equation}
%{\color{blue}{Note that $\psi$ is automatically an embedding.}}

 The main results, Theorems~\ref{Sophie2} and~\ref{Sophie3}, in this section give both necessary and sufficient conditions for the existence of an embedding $\phi : Z \to Y$ which factors through a given cover $\pi : X \to Y$ of the irreducible sofic shift $Y$.

  There are two main ideas in the proof of Theorem~\ref{Sophie2} which handles the case $h(Z) < h(Y)$, and both are similar to the main ideas in MacDonald's proof of Theorem \ref{Sophie}. The first idea is in Lemma~\ref{lem:large_entropy)injectively_fact_SFT}, which states that, given any $\epsilon > 0$, there is an irreducible SFT $W\subseteq X$  such that the restriction of $\pi$ to $W$ is injective and $h(W) > h(Y) - \epsilon$.
 The second idea is in Lemma~\ref{09-06-24}, which states that one can enlarge such a $W$ to include any periodic point $u$ in $X$ which has the same least period as that of $\pi(u)$ and $\pi(u) \not\in \pi(W)$ (and the new $W$ is still an irreducible SFT and the restriction of $\pi$ to the new $W$ is still injective).
 Then to complete the proof of Theorem~\ref{Sophie2}, one iteratively applies Lemma~\ref{09-06-24} to enlarge $W$ so that it has enough periodic points to accommodate an embedding $\psi: Z \to  W \subseteq X$, thereby obtaining an embedding $\phi : Z \to Y$ which factors through $\pi$. Finally, to handle the case $h(Z) = h(Y)$ we use a recent result from   \cite{Meye} {and \cite{PS}}, to obtain Theorem \ref{Sophie3}.

We remark that the first of the ideas above is related to an old result~\cite{MPW} which shows that there is an irreducible sofic shift $W'\subseteq X$ such that the restriction of $\pi$ to $W'$ is finite-to-one and surjective (and therefore satisfies $h(W') = h(Y)$).

The following somewhat technical lemma concerns  a quantitative estimate that involves forbidding long repetitions of a word in an irreducible sofic shift. 

\begin{lemma}\label{lem:eliminating_rep_pattern}
	Let $\mathcal{G}=(G,\mathcal{L})$ be a finite irreducible labeled graph such that $h(X_{\mathcal{G}})>0$, and let $y \in X_{\mathcal{G}}$ be a periodic point with least period $\ell$ and let $\overline{u} = y_0\ldots y_{\ell-1} \in \mathcal{B}_\ell(X_\mathcal{G})$. Let $\epsilon > 0$. Then, there exists $n\in \mathbb{N}$ such that for every pair of vertices $v_1,v_2$ in $G$ there exists $ i \in \{0,\ldots,\p(X_G)-1\}$ such that,
	\[
	\liminf_{j \to \infty}\frac{1}{j\p(X_G)+i}\log (\#\mathcal{L}(S_{j\p(X_G)+i}(v_1,v_2,n))) \ge h(Y)-\epsilon,
	\]
	where for $j \in \N$, $S_{j}(v_1,v_2,n)$ is the set of edge paths $w$ in $G$ of length $j$ from the vertex $v_1$ to the vertex $v_2$ such that $\overline{u}^{2n}$ does not occur as a subword of $\mathcal{L}(w)$.
\end{lemma}
\begin{proof}
	Let $\overline{u}$ be as in the statement.
	We first claim that there exists $k \in \N$ such that for every $n,j \in \N$ with $n,j > k$, and any $w \in \mathcal{B}_j(X_G)$ there exists $w' \in \mathcal{B}_j(X_G)$ such that the following holds:
	\begin{itemize}
		\item[(1)] The edge path $w'$ has the same initial and terminal vertices as $w$;
		\item[(2)] The word $\overline{u}^{2n}$ does not occur as a subword of $\mathcal{L}(w')$;
		\item[(3)] $w'$ differs from $w$ by at most $\lceil \frac{k}{n}j \rceil$ coordinates.%, where $\lceil \frac{k}{n}j \rceil$ the largest natural number less or equal to $\frac{k}{n}j$.
	\end{itemize}

	To prove the claim, using the fact that $G$ is an irreducible graph and $X_{\mathcal{G}}$ has positive entropy, there exists $k \in \N$ with the following property: For every pair of vertices $v_1,v_2$ in $G$, if  there exists an edge path of length $k \ell$ labeled by $\overline{u}^k$ from $v_1$ to $v_2$, then there exists another edge path from $v_1$ to $v_2$ with the same length $k \ell$ but having $\mathcal{L}$-labeling different from $\overline{u}^k$. Take any  $n,j \in \N$ with $n,j > k$. 
	
	Now given a word $w \in \mathcal{B}_j(X_G)$ (corresponding to an edge path of length $j$ in the graph $G$), for each maximal subpath of $w$ labeled by $\overline{u}^r$ for some $r \ge n$, 
	replace each subpath of length $k \ell$ starting at coordinates $0, n \ell, 2n \ell, \ldots , \lfloor \frac{r}{n} \rfloor n\ell -n\ell$ with a different edge path of length $k \ell$ having the same initial and terminal vertices but whose $\mathcal{L}$-labeling is different from $\overline{u}^k$.
	This is well-defined because $\ell$, the length of $\overline{u}$, is the least period of $\overline{u}^\infty$, so maximal intervals $I$ on which $\mathcal L(w)_I = \overline{u}^r$ for some $r$ are disjoint.
	
	We claim that $w'$ satisfies the requirements (1), (2) and (3).  Items (1) and (3) are clear.
	As for item (2), if $\mathcal{L}(w')_{[i, i+2n\ell)} = \overline{u}^{2n}$ for some $i$, then for some $j = i ~mod ~\ell$, $i \le j < i +n$, we have
	$\mathcal L(w)_{[j, j + n\ell)} = \overline{u}^{n}$, in which case {$\mathcal L(w')_{[j, j + n\ell)} \ne \overline{u}^{n}$}, a contradiction to 
	$\mathcal L(w')_{[i, i+2n\ell)} = \overline{u}^{2n}$.
	
	%The resulting word $w'$ satisfies the requirements (1) (2) and (3) and therefore the claim is proved.
	
	Thus, if $\mathcal{B}_{j,v_1,v_2}(X_G)$ is the set of admissible edge paths in $X_G$ of length $j$  whose initial vertex is $v_1$ and terminal vertex is $v_2$, then for every $w\in \mathcal{B}_{j,v_1,v_2}(X_G)$, we can obtain a $w'\in S_j(v_1,v_2,n)$ by adjusting at most $\lceil \frac{k}{n}j \rceil$ many coordinates. Note that $\mathcal{L}(w)$ and $\mathcal{L}(w')$ also differ by at most $\lceil \frac{k}{n}j \rceil$ many coordinates. Thus, each $\mathcal L(w')$ corresponds to at most  $(\# E_G)^{\lceil \frac{k}{n}j\rceil}$ many distinct $\mathcal{L}(w)$'s where $E_G$ is the set of all edges of $G$. Hence,
	\[
	\# \mathcal{L}(\mathcal{B}_{j,v_1,v_2}(X_G))  \le (\# E_G)^{\lceil \frac{k}{n}j\rceil } (\# \mathcal{L}(S_j(v_1,v_2,n))).
	\]
	
	For each pair of vertices $v_1, v_2$, let $i:=i(v_1,v_2)$ be the unique integer between $0$ and $\p(X)-1$ such that all paths from $v_1$ to $v_2$ in $G$ are of length $i (\bmod \p(X_G))$. Choose $n$ large so that $\log (\# E_G)^{\lceil \frac{k}{n}\rceil }<\epsilon$. Then, 
	\begin{align*}
		& \quad \liminf_{j\to \infty} \frac{1}{j\p(X_G)+i} \log (\# \mathcal{L}(S_{j\p(X)+i}(v_1,v_2,n))) \\
		&\geq \liminf_{j\to \infty} \frac{1}{j\p(X_G)+i} \Large\left[\log (\# \mathcal{L}(\mathcal{B}_{j\p(X_G)+i, v_1, v_2}(X_G))) \right.\\
		& \ \ \ \ \ \ \ \ \ \ \ \ \ \ \ \ \ \ \ \ \ \ \ \ \ \ \ \ \ \left.-\log \left((\# E_G)^{\lceil \frac{k}{n}(j\p(X_G)+i)\rceil }\right)\right] \\
		 &\geq h(Y)-\epsilon,
	\end{align*}
where the last inequality is true because there exists a positive integer $M$  such that for any pair of vertices $v_1', v_2'$ of $G$, 
$$
\# \mathcal{L} (\mathcal{B}_{j\p(X_G)+i', v_1', v_2'} (X_G)) \leq  \# \mathcal{L} (\mathcal{B}_{(j+M)\p(X_G)+i, v_1, v_2} (X_G)).
$$
where $i'=i(v_1', v_2')$. 
	\end{proof}

	\begin{lemma}\label{lem:large_entropy)injectively_fact_SFT}
		Let $\pi:X \to Y$ be a factor code from an irreducible subshift of finite type $X$ onto a sofic shift $Y$ of positive entropy. Then for every $\epsilon >0$ there exists an irreducible SFT $W \subseteq X$ such that $\pi\vert_W:W \to Y$ is injective, $h(W) \ge h(Y) -\epsilon$ and the global period of $W$ is the same as that of $X$.
	\end{lemma}
	\begin{proof}
		By Lemma \ref{08-05-24c}, we can assume that $X=X_G$ for some finite irreducible graph $G$ and that $\pi=\mathcal{L}_\infty$ for some labeling function $\mathcal{L}$ of the edges of $G$ by the letters of the alphabet of $Y$.
		
		Fix $\epsilon >0$. We assume without loss of generality that $\epsilon < h(Y)$. 
		
		Choose some  periodic point $x \in X$ with least period $\ell$. Suppose that $x_{[1,\ell]}=u \in \mathcal{B}_\ell(X)$, where $u=u_1\ldots u_\ell$, and let $\overline{u} = \mathcal{L}(u) \in \mathcal{B}_\ell(Y)$ be the $\mathcal{L}$-labeling of $u$ (note that $\ell$ must be a multiple of $\p(X)$).
		We also choose $n$ such that the statement of  Lemma~\ref{lem:eliminating_rep_pattern} holds.
		 %\begin{color}{green} The $n$ of the last lemma depends on a periodic point $y$ in $Y$, and at least a priori also on the choice of two vertexes in $G$. If you choose $n$ at this point you need to specify $y$ and the pair of vertexes.\end{color}

%		For every $n \in \mathbb{N}$, let $Y_n \subseteq Y$ be the subshift obtained from $Y$ by forbidding the word $\overline{u}^n \in \mathcal{B}_{n \ell}(W)$.

%		It is well known that entropy loss by forbidding a single word in an irreducible sofic shift goes to zero as the length of the forbidden word increases to $\infty$  (see for instance \cite{L}), so  $\lim_{n \to \infty}h(Y_n)= h(Y)$.
%		Choose $n \in \mathbb{N}$ such that
%		$h(Y_{n}) \ge h(Y) - \frac{1}{2}\epsilon$.
		We claim that there exists $v\in \mathcal{B}(X)$ such that the following holds:
		\begin{enumerate}
		\item[(i)] $\overline{u}^k$ appears exactly once in $\mathcal{L}(v)$, where $k$ is the maximum $j\in \mathbb{N}$ such that $\overline{u}^j$ appears in $\mathcal{L}(v)$ and $k\geq 2n+2$;
		\item[(ii)] the $\ell$-prefix and the $\ell$-suffix of $\mathcal{L}(v)$ are not cyclic permutations of $\overline{u}$.
		\end{enumerate}	
		To prove this claim, first note that there exists a word $a\in \mathcal{B}_\ell(X)$ such that $\mathcal{L}(a)$ is not a cyclic permutation of $\overline{u}$ because $Y$ has positive entropy. Since $X=X_G$ is irreducible, there exist $b_1,b_2\in \mathcal{B}(X)$ such that $ab_1 u^i b_2a \in \mathcal{B}(X)$ for  all $i$. Choose $v:= ab_1u^i b_2a$ for large enough $i$. Then it is not difficult to see that $v$ satisfies (i) and (ii).
		
		Denote $m:= \vert v\vert$ and let $s$ and $t$ be the initial and terminal vertices of $v$, respectively.
		%Because $X=X_G$ is an irreducible SFT, there exists letters of the alphabet of $X$ $u_0 \ne u_\ell$ and $u_{\ell+1} \ne u_1$ such that $v := u_0u^{2n} u_{\ell+1} \in \mathcal{B}_{2n\ell+2}(X)$,\begin{color}{green} I think I can easily find examples where this is not true. \end{color} and let $m = 2n\ell +2$. Then the word $v$ cannot self-overlap by more then one letter. \begin{color}{green} Please specify exactly what you mean here.\end{color}
		
%		Given $j \in \N$ let
%		\[B_j := \left\{ \mathcal{L}(w) \in \mathcal{B}_j(Y_n)~:~
%		u_{\ell+1}wu_{0} \in \mathcal{B}_{j+2}(X)
%		\right\}.\]
%		Then
%		{\color{cyan}{COMMENT: I believe the following growth rate statement but am unclear about the details of a proof}}
%		\[
%		\lim_{j\to \infty}\frac{1}{j \p(X)}\log (\# B_{j \p(X)})= h(Y_n) > h(Y)-\frac{1}{2}\epsilon.
%		\]
		%Fix $u_0$ and $u_{\ell}$, both of which are vertices of $G$. \begin{color}{green} In my book $u_0$ and $u_{\ell}$ are edges in $G$; not vertices.\end{color}
		For each $j$, let $S_j(t, s, n)$ be defined as in Lemma~\ref{lem:eliminating_rep_pattern}. Then, by Lemma~\ref{lem:eliminating_rep_pattern}, we can find $N \in \N$  sufficiently large so that
		\[
		\frac{1}{m+N}\log (\# \mathcal{L}(S_N(t, s, n))) \ge h(Y) -\epsilon.
		\]
		and also
		\[
		\frac{1}{m+N}\log (\# \mathcal{L}(S_{N-\p(X)}(t, s, n))) \ge h(Y) -\epsilon.
		\]

		For every $1 \le j \le N$ and any $w \in \mathcal{L}(S_j(t, s,n))$ there exists $ \tilde{w} \in S_j(t, s,n)$  such that $\mathcal{L}(\tilde w) =w$,  and noting that $v$ is a path in $G$ from $s$ to $t$, we have $v\tilde wv \in \mathcal{B}(X)$. %\begin{color}{green} This looks suspicious to me, but since I'm not sure which vertices in $G$ you fix at this point, I'm just confused. \end{color}
		Choose for every $1 \le j \le N$ a function $\Psi_j:\mathcal{L}(S_j(t, s,n))\to S_j(t, s,n)$ such that $\mathcal{L}(\Psi_j(w)) =w$ and $v\Psi_j(w)v \in \mathcal{B}(X)$.
		
		We now define the subshift $W \subseteq X$. A point $x \in X$ is in $W$ if and only if it  satisfies the following constraints:
		\begin{enumerate}
			\item The word $v$ occurs in $x$ with gaps bounded by $N$.
			\item Whenever $\tilde{w} \in \mathcal{B}_j(X)$ is a word that occurs between two consecutive occurrences of $v$, then $\mathcal{L}(\tilde{w}) \in \mathcal{L}(S_j(t, s,n))$ and $\Psi_j(\mathcal{L}(\tilde{w}))=\tilde{w}$.
		\end{enumerate}
		
		Then $W$ is a subshift of finite type because the new constraints are on words of bounded length, and it is clearly irreducible.
		
		%{\color{red}We will now show that  $\pi\mid_W:W \to Y$ is injective. This is because $\mathcal{L}(v)$ is a magic word for $\pi\vert_W$. (Comment: not quite right because we are not in the finite-to-one case. We seem to need more argument for the set of images of connectors $w$.)} 
		If $x,x' \in W$ and $y:= \pi(x') = \pi(x)$ then the word $\mathcal{L}(v)$ occurs in $\pi(x')=\pi(x)$ with gaps at most $N$, and the occurrences of  the word $\mathcal{L}(v)$ precisely correspond to the occurrences of $v$ in both $x$ and $x'$, because $v$ satisfies (i) and (ii). Also, any occurrence of  $w \in \mathcal{B}_j(\mathcal{L}_{\infty}(W))$ for $1\le j \le N$ between two consecutive occurrences of the word $\mathcal{L}(v)$  corresponds to an occurrence of $\Psi_j(w)$ in both $x$ and $x'$. This shows that $x=x'$, so $\pi\vert_W$ is  injective.
		
		We now show that $\p(W)=\p(X)$. Since $W \subseteq X$ we have that $\p(W)$ is a multiple of $\p(X)$.
		By the choice of $N$, there exist words $ w\in \mathcal{L}(S_N(t, s, n))$ and $w' \in \mathcal{L}(S_{N-\p(X)}(t, s, n))$. So the points
		$x = (v\Psi_N(w))^\infty$ and $x'=(v\Psi_{N-\p(X)}(w'))^\infty$ are both periodic points in $W$, having periods $m+N$ and $m+N-\p(X)$ respectively. Hence the gcd of the lengths of periodic points in $W$ divides $\p(X)$, which proves that $\p(W)=\p(X)$.
		
		It remains to check that $h(W) \geq h(Y) -\epsilon$.
		Indeed, for every $i \in \N$ and every $w^{(1)},\ldots,w^{(i)} \in \mathcal{L}(S_N(t, s, n))$ we have that
		\[
		v\Psi_N(w^{(1)})v\Psi_N(w^{(2)})\ldots v\Psi_N(w^{(i)}) \in \mathcal{B}_{i(N+m)}(W).
		\]
		So
		\[
		\log\left(\# \mathcal{B}_{i(m+N)}(W)\right) \ge i \log(\# \mathcal{L}(S_N(t, s, n))) \ge i(m+N) (h(Y) -\epsilon).
		\]
		Dividing by $i(N+m)$ and taking $i \to \infty$ we conclude that $h(W) \ge h(Y) -\epsilon$.
	\end{proof}

\begin{lemma}\label{09-06-24} Let $\mathcal G =(G,\mathcal{L})$ be a labeled graph with $G$ irreducible. Let $T_0 \subseteq X_G$ be an irreducible SFT of positive entropy such that $\mathcal{L}_\infty : T_0 \to X_\mathcal G$ is injective. Let $u \in X_G$ be a periodic point such that the least period of $u$ is the same as that of $\mathcal{L}_\infty (u)$ and $\mathcal L_\infty(u) \notin \mathcal L_\infty(T_0)$.	
	Then, there is an irreducible SFT $T \subseteq X_G$ such that $T_0 \cup \{u\} \subseteq T$ and $\mathcal{L}_\infty: T \to X_\mathcal G$ is injective.
\end{lemma}
\begin{proof}  By Lemma \ref{08-05-24d} we may assume that there is an irreducible subgraph $H$ of $G$ such that $T_0 = X_H$. Write $u=p^\infty$ where $p$ is a cycle in $G$. Since $\mathcal L_\infty(u) \notin \mathcal L_\infty(X_H)$ we can choose $R \in \mathbb N$ such that
\begin{equation}\label{18-01-25}
\mathcal L(p^R) \notin \mathcal {B}(\mathcal L_\infty(X_H)).
\end{equation}
Since $\mathcal L_\infty(X_H)$ has positive entropy by assumption, there are finite paths $I$ and $J$ in $G$ such that $Ip^kJ\in \mathcal {B}(X_G)$ for all $k \in \mathbb N$,
\begin{equation}\label{18-01-25a}
\mathcal L(Ip) \notin \mathcal {B}(\mathcal L_\infty(p^\infty))
\end{equation}
and
\begin{equation}\label{18-01-25b}
\mathcal L(pJ) \notin \mathcal {B}(\mathcal L_\infty(p^\infty)).
\end{equation}
Since $G$ is irreducible we can prolong $I$ to arrange that $I$ starts where $J$ ends; in fact, we can arrange that there is a vertex $v$ in $H$ such that $i_G(I) = t_G(J) =v$. Let $T \subseteq X_G$ be the elements of $X_G$ that present bi-infinite paths in $G$ obtained by concatenation of paths in $G$ of the following types:

    \begin{itemize}
    \item[(a)] finite paths in $H$ starting and ending at $v$, and of length exceeding $4 (|I| + |J| + |p^R|)$,
    \item[(b)] finite paths of the form $Ip^kJ$ for some $k \geq R +|I| + |J|$,
    \item[(c)] left-infinite paths in $H$ ending at $v$,
    \item[(d)] the left-infinite path $p^\infty J$,
    \item[(e)] right-infinite paths in $H$ starting at $v$,
    \item[(f)] the right-infinite path $I p^\infty$,
    \item[(g)] bi-infinite paths in $H$, and
    \item[(h)] the bi-infinite path $p^\infty$.
    \end{itemize}
 We claim that $T$ has the required properties. That $T$ is an SFT can be seen by checking that all words in $\mathcal {B}(T)$ of length $\geq 4(|I| + |J| + R|p|)$ are synchronizing for $T$, and $T$ is irreducible by construction and because $H$ is irreducible.  To see $\mathcal L_\infty : T \to X_\mathcal G$ is injective, let $z \in T$. Thanks to \eqref{18-01-25}, \eqref{18-01-25a} and \eqref{18-01-25b}, the occurrences of the word $\mathcal L(p)^{R+|I|+|J|}$ in $\mathcal L_\infty(z)$ give away the intervals in $z$ of the form (b),(d),(f),(h). All these intervals are determined by their image under $\mathcal L$ because $u$ and $\mathcal L(u)$ have the same least period by assumption. The remaining intervals are either bi-infinite or represent finite or infinite paths in $H$ starting and /or ending at the vertex $v$. These intervals are determined by their images under $\mathcal L$ because $\mathcal L_\infty$ is injective on $X_H$ by assumption, or in the case where $z$ represents the infinite path $p^\infty$, because $\mathcal L_\infty$ is injective on the orbit of $u$.
\end{proof}

In our next result, we iterate the procedure from Lemma~\ref{09-06-24} to enlarge an irreducible SFT $W$ embedded in $X$, on which the restriction of $\pi:X \to Y$   is injective, to a new irreducible SFT $W$ which contains several additional periodic points, and the restriction is still injective.

\begin{thm}\label{08-05-24g} Let $\pi : X \to Y$ be a factor code from an irreducible subshift of finite type $X$  onto a sofic shift $Y$ of positive entropy. Let $p \in \mathbb N$ be the global period of $X$. For each $\epsilon  > 0$ there is an $N_\epsilon \in \mathbb N$ such that for all $M \in \mathbb N$ there is an irreducible SFT $W \subseteq X$ of global period $p$ with the following properties:
	\begin{enumerate}
		\item $\pi : W \to Y$ is injective,
		\item $q_{np}(W) \geq e^{np(h(Y)-\epsilon)} \mbox{ when } \ n \geq N_\epsilon$,
		\item $q_{np}(W) \geq r_{np}(\pi) \mbox{ when }\ n \leq M$.
	\end{enumerate}

\end{thm}
\begin{proof}
Let $\pi:X \to Y$ be as in the statement, and fix $\epsilon >0$.
By Lemma \ref{lem:large_entropy)injectively_fact_SFT} we can find an irreducible SFT $\tilde{W} \subseteq X$ with $\p(\tilde{W})=p$ such that $\pi|_{\tilde{W}}$ is injective and $h(\tilde{W}) > h(Y) -\epsilon$. Since $h(\tilde{W}) = \lim_{n \to \infty} \frac{1}{np} \log q_{np}(\tilde{W})$, there exists $N_\epsilon \in\N$ such that $q_{np}(\tilde{W}) \geq e^{np(h(Y)-\epsilon)}$ when $n \geq N_\epsilon$. Given $M \in \N$ we can apply Lemma \ref{09-06-24} at most $\sum_{k=1}^M r_{kp}(\pi)$ times to get an irreducible SFT $W$ such that $\tilde{W} \subseteq W \subseteq X$, $\pi|_W$ is injective and (3) holds. Since $\tilde{W} \subseteq W$, it follows that (2) also holds.	
\end{proof}

\begin{remark}\label{24-05-24} We note that item (1) and item (3) in Theorem \ref{08-05-24g} imply that $q_{np}(W) = r_{np}(\pi)$ for $n \leq M$.
\end{remark}

The following result is of interest in its own right. The sufficiency part of it is a version of Krieger's theorem, Theorem \ref{Krieger}, where the target is only assumed to be irreducible, and not necessarily mixing. To formulate it we use the following notion of period for subshifts.

\begin{defn}
\label{p-periodic}
Let $p \in \mathbb N \backslash \{0\}$. A subshift $Z$ is called \emph{$p$-periodic} when there is a partition $Z = \sqcup_{i=0}^{p-1} Z_i$ of $Z$ into clopen sets $Z_i$ such that $\sigma(Z_i) = Z_{i+1}$ for all $i$, where the addition $i+1$ is taken modulo $p$. Such a partition will be referred to as a \emph{cyclic partition} of $Z$.
\end{defn}

We note that when $Z$ is $p$-periodic and $q \in \mathbb N$ divides $p$, then $Z$ is also $q$-periodic. In particular, $Z$ can be $p$-periodic for a natural number $p$ smaller than the period of $Z$ as defined in Definition 4.5.4 of \cite{LM}, i.e., the gcd of the periods of the periodic points in $Z$.

\begin{prop}\label{prop:irreducible_emb}

	Let $W$ be an irreducible SFT with global period $p$, and let $Z$ be any subshift. Then $Z$ embeds into $W$ if and only if $Z$ is conjugate to $W$ or all of the following conditions hold:
	\begin{itemize}
	\item[(1)] $h(Z) < h(W)$,
	\item[(2)] $Z$ is $p$-periodic, and
	\item[(3)] $q_{np}(Z) \leq q_{np}(W)$ for all $n \in \mathbb N$.
	\end{itemize}

\end{prop}

\begin{proof} (Sufficiency:) If $Z$ is conjugate to $W$ then $Z$ trivially embeds into $W$. Assume $Z$ is not conjugate to $W$.  Let $W = \sqcup_{i=0}^{p-1} W_i$ be the canonical cyclic partition of $W$. By assumption (2) there is also a cyclic partition $Z = \sqcup_{i=0}^{p-1} Z_i$ of $Z$. The topological entropy of $(Z_0,\sigma^p)$  is equal to $p$ times the topological entropy of $(Z,\sigma)$, and the topological entropy of $(W_0,\sigma^p)$ is equal to $p$ times the topological entropy of $(W,\sigma)$. It follows therefore from (1) that $h(Z_0,\sigma^p) < h(W_0,\sigma^p)$. The number of $\sigma^p$-orbits of size $n$ in $Z_0$ is equal to the number of $\sigma$-orbits in $Z$ of size $pn$, and similarly for $(W_0,\sigma^p)$ and $(W,\sigma)$. Since $(W_0,\sigma^p)$ is a mixing SFT, it follows therefore now from (3) that we can use Krieger's embedding theorem, Theorem \ref{Krieger}, to get an embedding of $(Z_0,\sigma^p)$ into $(W_0,\sigma^p)$; that is, there is an injective map $\iota : Z_0 \to W_0$ such that $\iota \circ \sigma^p = \sigma^p \circ \iota$. Define $j : Z \to W$ such that
$$
j(z) := \sigma^i \circ \iota \circ \sigma^{-i}(z), \ z \in Z_i,
$$
for $i \in \{0,1,\cdots , p-1\}$. Then $j : (Z,\sigma) \to (W,\sigma)$ is an embedding of $Z$ into $W$.

(Necessity:) We now prove that if $Z$ is not conjugate to $W$ and $Z$ embeds into $W$, then conditions (1) (2) (3) hold. The proof that conditions (1) and (3) are necessary is identical to the well-known argument in the mixing case, see for example, \cite[Page 338-339]{LM}. Now let us check that (2) holds. Let $j: (Z, \sigma)\to (W,\sigma)$ be an embedding and let  $W = \sqcup_{i=0}^{p-1} W_i$ be the canonical cyclic partition of $W$ and define
$$
Z_i:= j^{-1}(j(Z)\cap W_i)
$$
for each $i$. Clearly $Z= \sqcup_{i=0}^{p-1} Z_i$. Moreover, the $Z_i$'s are clopen since the $W_i$'s are. Then, since
$$
\sigma(Z_i)=\sigma(j^{-1} (j(Z)\cap W_i)) = j^{-1} (\sigma(j(Z))\cap \sigma(W_i)) = j^{-1} (j(Z)\cap W_{i+1})
$$
under addition modulo $p$, we have $\sigma(Z_i)= Z_{i+1}$ (modulo $p$) and therefore (2) holds.
\end{proof}

\begin{remark} It follows from \cite[Theorem 10.1.1]{LM} that when $Z$ and $W$ are both irreducible SFTs and $h(Z) < h(W)$, there will be an embedding of $Z$ into $W$ when
\begin{equation}\label{15-06-25b}
q_n(Z) \leq q_n(W)
\end{equation}
for all $n \in \mathbb N$. This follows also from  Proposition \ref{prop:irreducible_emb} since condition \eqref{15-06-25b} implies that the global period of $Z$ will be a multiple of $p$ so that $Z$ will be $p$-periodic.
\end{remark}

Now we arrive at the generalization of MacDonald's Theorem to irreducible sofic targets, which is a main result of this section.

\begin{thm}\label{Sophie2} Let $\pi : X \to Y$ be a factor code from an irreducible SFT $X$ of global period $p$ to a sofic shift $Y$. Let $Z$ be a subshift such that $h(Z) < h(Y)$. Then, there is an embedding from $Z$ into $Y$ which factors through $\pi$ if and only if $Z$ is $p$-periodic and $q_{np}(Z) \leq r_{np}(\pi)$ for all $n \in \mathbb N$.
\end{thm}
\begin{proof} We first explain the necessity of the conditions. Assume that   there are embeddings $\phi : Z \to Y$ and $\psi : Z \to X$ such that $\phi=\pi \circ \psi$.  Then by the necessity part of Krieger's embedding theorem for irreducible SFT targets, Proposition \ref{prop:irreducible_emb}, $Z$ must be $p$-periodic. The bounds $q_{np}(Z) \leq r_{np}(\pi)$ for $n \in \mathbb N$ follow immediately from the fact that for every $n \in\N$,  $\phi= \pi \circ \psi$ induces an injective function from $Q_{np}(Z)$ into $\pi(Q_{np}(X)) \cap Q_{np}(Y)$ .

  For sufficiency of the two conditions, assume that $Z$ is $p$-periodic and that $q_{np}(Z)\leq r_{np}(\pi)$ for all $n\in \mathbb N$. Set $\epsilon := \frac{1}{2}(h(Y)-h(Z))$. Let $X= \sqcup_{i=0}^{p-1} X_i$ be the canonical cyclic partition of $X$. Since $\limsup\limits_{n\to \infty} \frac{1}{np} \log (q_{np}(Z)) \leq h(Z) < h(Y)- \epsilon$, there is an $N_1 \in \mathbb N$ such that
\begin{equation}\label{22-05-24}
	q_{np}(Z) \leq e^{np(h(Y)- \epsilon)}
\end{equation}
for all $n \geq N_1$. It follows from Theorem \ref{08-05-24g} that there exist an $N_\epsilon\in \mathbb N$ and an irreducible SFT $W \subseteq X$ of global period $p$ such that
\begin{itemize}
	\item[(i)] $\pi : W \to Y$ is injective,
	\item[(ii)] $q_{np}(W) \geq e^{np(h(Y)-\epsilon)}, \ n \geq N_\epsilon$,
	\item[(iii)] $q_{np}(W) \geq r_{np}(\pi), \ n \leq \max \{N_\epsilon, N_1\}$.
\end{itemize}
Since $W$ is an irreducible SFT of global period $p$, $$
h(W) = \lim_{n \to \infty} \frac{1}{np} \log q_{np}(W) ,
$$
and hence (ii) implies that $h(W) \geq h(Y) - \epsilon > h(Z)$. By combining (ii) and \eqref{22-05-24} we find that $q_{np}(W) \geq q_{np}(Z)$ for all $n \geq \max \{N_\epsilon, N_1\}$. For $n \leq \max\{N_\epsilon, N_1\}$, it follows from the assumption and (iii) that
\begin{equation}\label{22-05-24a}
	q_{np}(W) \geq r_{np}(\pi) \geq q_{np}(Z) .
\end{equation}
We can then apply Proposition \ref{prop:irreducible_emb} to obtain an embedding $\iota : Z \to W$, and then $\pi \circ \iota : Z \to Y$ is an embedding of $Z$ into $Y$ which factors through $\pi$.

\end{proof}

The case $p=1$ of Theorem \ref{Sophie2} is Macdonald's theorem. In the case where $X=Y$ and $\pi$ is the identity map, Theorem \ref{Sophie2} gives back Proposition \ref{prop:irreducible_emb}.

%Using the notion of retracts for subshifts as introduced in \cite{Meye}, it is now easy to obtain the following strengthening of Theorem \ref{Sophie2} where there is no a priori assumption on entropies:

The following result refines Theorem~\ref{Sophie2} by giving a complete characterization of the subshifts $Z$ that admit an embedding that factors through a given factor code $\pi:X \to Y$ on an irreducible SFT, by dealing with the remaining case  that $h(Z)=h(Y)$. The refinement involves the notion of a \emph{retraction} between subshifts. The classical notion of a retraction from topology was  introduced into symbolic dynamics independently in \cite{Meye} and \cite{PS}:
\begin{defn}\label{def:retract}
Let $X$ be a subshift. A sliding block code $r:X \to X$ is called a \emph{retraction} if $r\circ r=r$.
A subshift $Y \subseteq X$ is called a \emph{retract} of $X$ if there exists a retraction $r:X \to X$ such that $Y=r(X)$. Note that in this case $r\mid_Y = \id_Y$.
\end{defn}

\begin{thm}\label{Sophie3} Let $\pi : X \to Y$ be a factor code from an irreducible SFT $X$ of global period $p$ onto a sofic shift $Y$, and let $Z$ be an arbitrary subshift. The following conditions (1)-(3) are equivalent:
\begin{itemize}
\item[(1)] There is an embedding from $Z$ into $Y$ which factors through $\pi$.
\item[(2)] \begin{itemize}
    \item[(a)] There is a sliding block code $\rho : Y \to X$ such that $\pi \circ \rho = \id_Y$ and $Z$ is conjugate to $Y$, or
    \item[(b)] $Z$ is $p$-periodic, $h(Z) < h(Y)$ and $q_{np}(Z) \leq r_{np}(\pi)$ for all $n \in \mathbb N$.
    \end{itemize}
\item[(3)]\begin{itemize}
    \item[(a')] $Z$ and $Y$ are conjugate SFTs, and there is an embedding $\rho : Y \to X$ such that $\rho \circ \pi$ is a retraction, or
    \item[(b)] $Z$ is $p$-periodic, $h(Z) < h(Y)$ and $q_{np}(Z) \leq r_{np}(\pi)$ for all $n \in \mathbb N$.
    \end{itemize}    
\end{itemize}
\end{thm}

\begin{proof} We first prove that (a) $\Leftrightarrow$ (a'): Assume (a). Then $r:= \rho \circ \pi : X \to X$ is a retraction. As proved in  \cite[Proposition 4.10]{Meye} and also in \cite[Lemma 4.3]{PS}, any  retract of an SFT is also an  SFT. Thus $r(X)$ is an SFT. Since $\rho$ is a conjugacy of $Y$ onto $r(X)$, (a') follows.  Assume (a'). Then $\rho \circ \pi \circ \rho \circ \pi = \rho \circ \pi$. Since $\pi$ is surjective and $\rho$ injective, it follows that $\pi \circ \rho = \id_Y$. Hence (a) holds.

We have shown that (2) $\Leftrightarrow$ (3) and therefore it suffices now to show that (1) $\Leftrightarrow$ (2). Assume first that (1) holds, i.e. there are embeddings $\phi : Z \to Y$ and $\psi : Z \to X$ such that \eqref{20-03-23cx} commutes. The existence of $\psi$ shows that $Z$ must be $p$-periodic. If $\phi$ is not surjective it follows from Corollary 4.4.9 of \cite{LM} that $h(Z) < h(Y)$ and then Theorem \ref{Sophie2} applies to give (b). Assume therefore that $\phi(Z) = Y$. Then $\phi$ is a conjugacy. Set $\rho := \psi \circ \phi^{-1} : Y \to X$ and note that $\pi \circ \rho = \id_Y$. It follows that (a) holds. 

For the converse, assume first that (a) holds. Let $\phi : Z \to Y$ be a conjugacy. Set $\psi:= \rho \circ\phi$ and note that \eqref{20-03-23cx} commutes. When we instead assume that (b) holds we get the desired embedding from Theorem \ref{Sophie2}.
\end{proof}

\begin{remark}\label{29-07-25} The conditions given in (a) and (a') are somewhat abstract and motivates the following question:  Given a factor map $\pi:X \to Y$ between irreducible SFTs $X$ and $Y$, when is there a sliding block code $\rho : Y \to X$ such that $\pi \circ \rho = \id_Y$, or equivalently, when does there exist an embedding $\rho:Y \to X$ (in terms of more  explicit, hopefully  checkable conditions) such that $\rho \circ \pi:X \to X$ is a retraction? A necessary condition is that every periodic point in $Y$ has a preimage which is a periodic point with the same period.
\end{remark}

\section{Factorizable embeddings}\label{facemb}

In Section \ref{Sophie5} we considered an irreducible sofic shift $Y$ and a given cover $\pi : X \to Y$ of $Y$ in order to decide when a given subshift $Z$ can embed into $Y$ via $\pi$. The answer was obtained in Theorem \ref{Sophie3}, and we consider therefore now a different but closely related problem. Given only $Y$ and $Z$, when can we find an embedding of $Z$ into $Y$ which factors through \emph{some} sliding block code $\pi : X \to Y$ out of an irreducible SFT X? And if we can, what properties can we ask of $\pi$?

We start with the following definition.
\begin{defn} \label{factorizable_def}
	Let $Z$ and $Y$ be subshifts. An embedding $\phi : Z \to Y$ is said to be \emph{factorizable} when $\phi$ factors through (see Definition \ref{factorthrough}) some irreducible SFT. That is, when there exist an irreducible SFT $S$ and a sliding block code $\pi: S\to Y$ such that
	the following diagram commutes for some embedding $\psi : Z \to S$:
	\begin{equation}\label{20-03-23c}
		\xymatrix{
			& S \ar[d]^-{\pi}  \\
			Z \ar[ur]^-\psi \ar[r]_-\phi & Y}
	\end{equation}
	Moreover, we say that $\phi$ is \emph{I-factorizable} when $\pi$ can be chosen to be injective (and hence an embedding), and that $\phi$ is \emph{S-factorizable} when $\pi$ can be chosen to be surjective (and hence a cover of $Y$).
\end{defn}

We first have the following observation regarding the proof of Theorem \ref{Sophie2}.

\begin{remark}\label{18-06-24e} The embedding $Z \to Y$ constructed in the proof of Theorem \ref{Sophie2} is $I$-factorizable (through the SFT $W$ therein) by construction, giving us the following implication in the setting of Theorem \ref{Sophie2}: If there is an embedding of $Z$ into $Y$ which factors through a factor code $\pi$, then there is also an $I$-factorizable embedding of $Z$ into $Y$. It follows that for an arbitrary subshift $Z$ and an arbitrary irreducible sofic shift $Y$ such that $h(Z) < h(Y)$, the implication (i) $\Rightarrow$ (ii) holds, where
	\begin{itemize}
		\item[(i)] There is an $S$-factorizable embedding $Z \to Y$.
		\item[(ii)]  There is an $I$-factorizable embedding $Z \to Y$.
	\end{itemize}
\end{remark}

\subsection{Receptive periodic points} \label{25-01-04-b}

Let $\mathcal G =(G,\mathcal L)$ be a right-resolving labeled graph.
%and $\mathcal{L}_\infty : X_G \to X_\mathcal G$ be the factor {\color{red} code} from the edge shift $X_G$ of $G$ onto the sofic shift $X_\mathcal G$ obtained by reading the labels.
A word $m \in \mathcal{B}(X_{\mathcal G})$ is \emph{magic} for $\mathcal G$ if all paths in $G$ labeled $m$ have the same terminal vertex
\footnote{Note on terminology: In \cite[Chapter 3]{LM}, such a word is called ``a synchronizing word for $\mathcal{G}$". In order to reduce confusion, we use the term ``magic" which comes from considerations in \cite[Chapter 9]{LM}.}.

\begin{lemma}\label{19-03-23d} Let $\mathcal G = (G, \mathcal{L})$ be the right Fischer cover of an irreducible sofic shift $Y$.  Then a word $u \in \mathcal{B}(Y)$ is synchronizing for $Y$ if and only if it is magic for $\mathcal G$.
\end{lemma}
\begin{proof} This follows from Exercise 3.3.3 of \cite{LM}, but we sketch a proof for the benefit of the reader. Recall that the right Fischer cover $\mathcal G = (G,\mathcal L)$  of an irreducible sofic shift $Y = X_{\mathcal G}$ is the unique, up to labeled graph isomorphism, irreducible, right-resolving, follower-separated presentation of $Y$.

 Assume that $w \in \mathcal B(X_\mathcal G)$ is synchronizing for $X_\mathcal G$. Consider two paths $u,v$ in $G$ with label $w$ and assume for a contradiction that they do not terminate at the same vertex. Since the right Fischer cover is follower-separated,
it follows that there is a path $l$ in $G$ starting at the terminal vertex of $u$, but no path starting at the terminal vertex of $v$ has the same label $\mathcal L(l)$ as $l$. (If this is not the case, it must be true with $u$ and $v$ interchanged and the following reasoning would just work as well.) It follows from \cite[Proposition 3.3.16]{LM} that there is a magic word for $\mathcal G$. Since $G$ is irreducible and the labeling is right-resolving, it follows that we can extend $l$ to the right to obtain a path $l'$ in $G$ which terminates at the initial vertex of $v$ and its label $\mathcal L(l')$ is magic for $\mathcal G$. Then $\mathcal L(l')w = \mathcal L(l'v) \in \mathcal B(X_\mathcal G)$, and $w\mathcal L(l') = \mathcal L(ul') \in \mathcal B(X_\mathcal G)$, but $\mathcal L(l')w \mathcal L(l') \notin \mathcal B(X_\mathcal G)$, since all paths labeled $\mathcal L(l')w$ must terminate at the vertex where $v$ terminates, but there is no path labeled $\mathcal L(l')$ which starts at that same vertex. This contradicts that $w$ is synchronizing for $X_\mathcal G$.

The converse, that magic implies synchronizing, is easy and left to the reader.
\end{proof}

%{\color{blue}{
%It is well known  that for a right-resolving labeled graph
%$\mathcal G =(G,\mathcal L)$ with $G$ irreducible, $\mathcal{L}_\infty$ is almost invertible if and only if it has a magic word. As mentioned earlier, the right Fishcer cover of an irreducible sofic shift is almost invertible and therefore has a magic word. Therefore an irreducible sofic shift always has a synchronizing word by the preceding Lemma.
%\newline
%Klaus says: The existence of a synchronizing word for a sofic shift, also one which is not irreducible, follows rather explicitly from Proposition 3.3.16 in [LM]; except for the existence of a right-resolving and follower separated presentation which I have been thinking of as 'well-known'. It follows from Theorem 3.3.2 and Lemma 3.3.8 in [LM]. But maybe this just goes to show how subjective the notion 'well-known' is.}}

\begin{defn} \label{receptive_def}
	A periodic point $x \in X$ of least period $p$ is \emph{receptive} (for $X$) when there are synchronizing words $m_1,m_2 \in \mathcal B(X)$ such that $m_1(x_{[0,p)})^km_2 \in \mathcal B(X)$ for all $k \in \mathbb N \backslash \{0\}$.
\end{defn}	
	This notion and the name were introduced by Mike Boyle in \cite{B}. We denote by $\Rec_n(X)$ the set of receptive periodic points of least period $n$, and set
$$
\rec_n(X) := \# \Rec_n(X).  
$$
%Also, %recall from Section \ref{Sophie5} that
%It is easy to see that any periodic point which contains a synchronizing word is receptive.

{Below are a few easy observations about receptive periodic points. 

\begin{prop} \label{receptive_other_characterization}
\label{easy}
Let $Y$ be an irreducible sofic shift. 
\begin{enumerate}
 \item 
A periodic point $y$ of least period $p$ is receptive iff for every positive integer $N$ there are synchronizing words $u,v$ in ${\mathcal B}(X)$ such that for all $k \ge N$ $u(y_{[0,p-1]})^kv \in {\mathcal B}(Y)$;
\item 
If a periodic point $y \in Y$ contains a synchronizing word in $Y$, then $y$ is receptive;
\item 
If $Y$ is an SFT, then all periodic points are receptive.
%\item 
%If $\psi:Y \to Z$ is a conjugacy and $y \in Y$ is receptive, then $\psi(y)$ is receptive.
\end{enumerate}
\end{prop}

\begin{proof}

(1): The ``only if''part  holds with $N=1$ by definition of receptive.  For the ``if'' part, given $u,v$ and $N$, let $u' = u(y_{[0,p-1]})^N$.  Then $u'$ is synchronizing and for all $k \ge 1$, $u'(y_{[0,p-1]})^kv = u(y_{[0,p-1]})^{k + N}v \in \mathcal{B}(X)$. 

(2): If $y_{[i,j]}$ is synchronizing, then so is any word which contains $y_{[i,j]}$  and thus so is $u:= y_{[0,mp-1]}$ for some positive integer $m$. But then for any $k \ge 1$, $u(y_{[0,p-1]})^k u$ is a power of $y_{[0,p-1]}$  and thus belongs to $\mathcal{B}(Y)$, and so $y$ is receptive. 

(3): This follows immediately from the fact that in an SFT all sufficiently long words are synchronizing (see, for example, \cite[Theorem 2.1.8]{LM}).
%4. Fill in this proof. 
\end{proof}

%We remark that the converse to Proposition~\ref{easy} (part (3)) is not true. 
%Boyle ~\cite[just before Corollary 3.7]{B} gave a simple example of a  
%mixing sofic shift $Y$ such that every periodic point is receptive but $Y$ is not an SFT. The example is essentially the even shift, with an extra ``free'' symbol that effectively does not introduce any new constraints.  
%Precisely,  $Y$ is the subshift on the alphabet $\mathcal{A} = \{0,1,2\}$ defined by forbidding the words $10^n1$, with $n$ odd; here, the 
%symbol 2 is synchronizing and for all allowed words $2u2$ is allowed (in particular for $u = (y_{[0,p-1]})^k$ for any periodic point $y \in Y$ and any $k$). And $Y$ is not an SFT because for all odd $n$, $10^n$ and $0^n1$ are allowed, but $10^n1$ is not allowed.}

%\begin{lemma}\label{18-03-23b} In an SFT all periodic points are receptive.
%\end{lemma}
%\begin{proof} Let $p \in Q_n(X)$. Since $X$ is an SFT there is an $M \in \mathbb N$ such that all words in $\mathcal B(X)$ of length $\geq M$ is synchronizing, cf. Proposition 2.1.7 and Theorem 2.1.8 in \cite{LM}. In particular, $p_{[0,Mn)}$ is synchronizing, and $p_{[0,Mn)}p_{[0,n)}^i p_{[0,Mn)} \in \mathcal B(X)$ for all $i \in \mathbb N$.
%\end{proof}

\begin{lemma}\label{11-03-23b} Let $X$ be an irreducbile SFT and  $\pi : X \to Y$ be a factor code. Let $w \in Y$ be a receptive periodic point. Then, there exist an irreducible SFT $X'$, a factor code $\pi' : X' \to Y$ and an embedding $\iota : X \to X'$ such that $\pi' \circ \iota = \pi$ and ${\pi'}^{-1}(w)$ contains a periodic element of the same least period as $w$. Moreover, when the least period of $w$ is a multiple of the global period of $X$, the global period of $X'$ can be arranged to be the same as that of $X$.
\end{lemma}
\begin{proof}
	Up to recoding, we may assume that $X=X_G$ for an irreducible graph and $\pi: X_G\to Y$ is a labeling of edges of $G$ (cf. Lemma \ref{08-05-24c}).  Let $n$ be the least period of $w$. Since $w$ is receptive, there are synchronizing words $s, t\in \mathcal {B}(Y)$ such that $s(w_{[0,n-1]})^k t\in \mathcal {B}(Y)$ for all $k>0$. Let $\alpha$ be a path in $G$ labeled $sw_{[0,n-1]} t$ and let $i_G(\alpha)$ and $t_G(\alpha)$ be the initial and terminal vertices of $\alpha$, respectively. Now, create a new graph $G'$ by adding the following labeled paths to $G$:
	\begin{enumerate}
		\item A simple path of length $\vert s \vert+n$ starting at $i_G(\alpha)$ and label the path by $s w_{[0,n-1]}$. Let $r$ be the terminal vertex of this path;
		\item A simple cycle of length $n$ at $r$. Label it by $w_{[0,n-1]}$.
		\item A simple path of length $\vert t\vert$ from $r$ to $t_G(\alpha)$. Label it by $t$.
	\end{enumerate}
	Let $X':=X_{G'}$, $\pi'$ be the $1$-block code defined by the labeling of $G'$ and $Y'$ be the factor of $X'$ under $\pi'$ (i.e., $\pi': X'\to Y'$). First, observe that by construction, $\pi^{-1}(w)$ contains a periodic element of least period $n$. Then, noting that $G$ is a subgraph of $G'$, we immediately derive that there is an embedding $\iota: X\to X'$; moreover, since $s$ and $t$ are synchronizing, the language of $Y'$ is the same as that of $Y$, implying that $Y'=Y$. Thus, $\pi' \circ \iota$ is just $\pi$.
	
	Now, let $p$ be the period of $G$ (and therefore also the global period of $X$) and assume that $n$ is a multiple of $p$.
    Consider the vertex $i_G(\alpha)$ in $G$. The global period $p$ of $G$ and hence of $X_G$ is the greatest common divisor of the lengths of cycles in $G$ starting and ending at $i_G(\alpha)$, and similarly for $G' $ and $X_{G'}$. Consider such a cycle in $G'$ of length $l$. If it contains some of the new vertices it must contain a number of the new paths from $i_G(\alpha)$ to $t_G(\alpha)$; all having lengths in $\{|s|+|t| + m n: m\in \mathbb{N}\}$. By exchanging each of these subpaths by $\alpha$ we get a cycle in $G$ whose length is in $\{l - m n: m\in \mathbb{N}\}$. Hence, when we assume that $n$ is a multiple of $p$ we conclude that $p$ divides $l$. It follows that the global period of $G'$ is also $p$.
\end{proof}

Lemma \ref{11-03-23b} leads to the following alternative characterization of receptive periodic points.

\begin{thm}\label{16-03-23a} Let $Y$ be an irreducible sofic shift and $w$ a periodic point in $Y$ with least period $p$. Then, $w$ is receptive if and only if there is an irreducible SFT $X$, a factor code $\pi : X \to Y$ and a periodic point $x\in X$ with least period $p$ such that $\pi(x) = w$.
\end{thm}
\begin{proof} `if part': By recoding (cf. Lemma \ref{08-05-24c}), we may assume that there is a labeled graph $(G, \mathcal{L})$ such that $G$ is irreducible, $X=X_G$ and $\pi$ is the labeling map defined by $\mathcal{L}$. Then $x_{[0, p-1]}$ is represented by a cycle $\gamma$ in $G$. Now
choose a synchronizing word $s$ in $Y$ and let $\tau$ be a path in $G$ labeled $s$. Since $G$ is irreducible, there are paths $\alpha$ and $\beta$ in $G$ such that for all $k$, $\tau \alpha \gamma^k \beta \tau$ is also a path in $G$. Let $u$ be the label of $\tau \alpha$ and $v$ be the label of $\beta\tau$. Then, $u$, $v$ are both synchronizing and $u w_{[0, p-1]}^k v$ is allowed in $Y$ for all $k$. Thus, $w$ is receptive.
%By Lemma \ref{18-03-23b}, all periodic points of $X$ are receptive; in particular, $x$ is receptive. Since $x$ and $w$ have the same least period it follows from \cite[Lemma 7.2 ]{T} that $w$ is receptive.
	
	`only if part': This follows from Lemma \ref{11-03-23b}.
\end{proof}

%{\color{blue}
%It is not hard to see that the image of a receptive periodic point under a factor code needs not be receptive. A simple example illustrating this is the factor map $\pi: X\to Y$ from the golden mean shift $X$ to the even shift $Y$, where $0^\infty\in Y$ is not receptive, but all its preimages, which are periodic points in $X$, must be receptive by item (3) of Proposition \ref{receptive_other_characterization}.
%However, as a consequence of Theorem \ref{16-03-23a}, the following corollary says that the image of a receptive periodic point will be receptive as long as the least period it preserved. Throughout the remainder of the paper, for a periodic point $y$, we use $\per(y)$ to denote the least period of $Y$.

%\begin{cor} \label{image_of_receptive}
%Let $Y$ be an irreducible sofic shift, $\psi: Y\to Z$ be a factor code and $y\in Y$ be a receptive periodic point. If $\per(y)=\per(\psi(y))$, then $\psi(y)$ is receptive. 
%In particular, conjugacies preserve receptivity.
%In particular, any conjugacy $\pi:Y \to Z$ maps the set of receptive periodic points onto the set of receptive periodic points of $Z$.
%\end{cor}

%\begin{proof}
%Since $y$ is receptive, by Theorem \ref{16-03-23a}, there is a cover $\pi: X\to Y$ and a periodic point $x\in X$ such that $\pi(x)=y$ and $\per(x)=\per(y)$. Let $\phi:= \psi \circ \pi$. Then, $\phi: X\to Z$ is a cover of $Z$. Moreover, $\phi(x)=\psi\circ \pi(x)=\psi(y)$ and $\per(x)=\per(y)=\per(\psi(y))$. Thus, by Theorem \ref{16-03-23a}, $\psi(y)$ is receptive.
%\end{proof}

%}

\begin{defn}\label{15-06-25}
Let $Y$ be an irreducible sofic shift.  We denote by $\text{per}(Y)$ the greatest common divisor of $\{n \in \mathbb N: \ \Rec_n(Y) \neq \emptyset\}$. In symbols
$$
\text{per} (Y) := \ \gcd\left( \{n \in \mathbb N: \ \Rec_n(Y) \neq \emptyset\}\right) .
$$
\end{defn}

When $Y$ is an irreducible SFT, {recalling from Proposition \ref{receptive_other_characterization} item (3) that all periodic points in $Y$ are receptive}, we have $\text{per}(Y)$ is the same as the global period of $Y$. However, in general $\text{per} (Y)$ is not the same as the gcd of the periods of the periodic points of $Y$. The above definition of $\text{per}(Y)$ will play a pivotal role in the following, but we postpone a more detailed investigation of the notion of period for irreducible sofic shifts to Section \ref{period}.

\begin{lemma}\label{05-07-25} Let $X$ and $Z$ be irreducible SFTs with canonical cyclic partitions $X = \sqcup_{i=1}^n X_i$ and $Z= \sqcup_{j=1}^m Z_j$, respectively. Let $Y$ be a sofic shift and let $\pi : X \to Y$ and $\rho:Z \to Y$ be factor codes. Assume $\# \pi^{-1}(y) = 1$ for some element $y \in Y$. Then, $n$ divides $m$ and there is a surjective map $\mu : \{1,2,\cdots, m\} \to \{1,2,\cdots, n\}$ such that $\rho(Z_j) = \pi(X_{\mu(j)})$ for all $j \in \{1,2,\cdots, m\}$.
\end{lemma}
\begin{proof} Note that $\sigma^{nm}$ is mixing and hence transitive on $\pi(X_i)$. Let $a_i \in \pi(X_i)$ be a point with a dense orbit for $\sigma^{nm}$. Since $\pi(X_i) = \bigcup_{j=1}^m \rho(Z_j) \cap \pi(X_i)$, there is a $j \in \{1,2,\cdots, m\}$ such that $\rho(Z_j) \cap \pi(X_{i})$ has non-empty interior in $\pi(X_i)$. It follows that $\sigma^{nm\ell}(a_i) \in \rho(Z_j) \cap \pi(X_{i})$ for some $\ell \in \mathbb Z$, implying that $\pi(X_i) \subseteq \rho(Z_j)$. A symmetric argument shows that $\rho(Z_j) \subseteq  \pi(X_{i'})$ for some $i' \in \{1,2,\cdots, n\}$.
By assumption there is a $k$ such that the set $\pi(X_k)$ contains an element with a unique pre-image under $\pi$. However,
\begin{equation}\label{05-07-25b}
\sigma(\pi(X_k)) = \pi(X_{k+1})
\end{equation}
with addition mod $n$, and hence $\pi(X_k)$ contains such an element for every $k$. The conclusion that  $\pi(X_i) \subseteq \rho(Z_j)  \subseteq  \pi(X_{i'})$ implies therefore that $i =i'$, and hence that $\pi(X_i) = \rho(Z_j)$.
 Since
\begin{equation}\label{05-07-25a}
\sigma( \rho(Z_j)) =  \rho(Z_{j+1})
\end{equation}
with addition mod $m$, it follows that for every $j \in \{1,2,\cdots,m\}$ there is an element $\mu(j) \in \{1,2,\cdots, n\}$ such that $\rho(Z_j) = \pi(X_{\mu(j)})$. Since each $\pi(X_i)$ contains an element with a unique pre-image under $\pi$, the element $\mu(j) \in \{1,2, \cdots, n\}$ is unique so that $\mu: \{1,2,\cdots, m\} \to \{1,2,\cdots, n\}$ is a well-defined map. It follows then from \eqref{05-07-25a} and \eqref{05-07-25b} that $c_n \circ \mu = \mu \circ c_m$, where $c_n$ and $c_m$ denote the cyclic permutations of $\{1,2,\cdots, n\}$ and $\{1,2,\cdots, m\}$, respectively. This in turn implies that $\mu$ is surjective and that $n$ divides $m$.
\end{proof}

\begin{cor}\label{05-07-25d} Let $Y$ be an irreducible sofic shift and $\pi : X \to Y$ the right Fischer cover of $Y$ and $\rho : Z \to Y$ an arbitrary cover of $Y$. Then $\text{per}(X)$ divides $\text{per}(Z)$.
\end{cor}
\begin{proof} This follows from Lemma \ref{05-07-25} because $\# \pi^{-1}(y) =1$ when $y$ is a doubly transitive point, as mentioned in Section~\ref{notations}.
\end{proof}

\begin{lemma}\label{03-06-25}  Let $Y$ be an irreducible sofic shift and $\pi : X \to Y$ a cover of $Y$. Then,
\begin{itemize}
\item[(a)] $\text{per} (Y)$ divides $\text{per} (X)$.
\end{itemize}
Assume that $\pi : X \to Y$ is the Fischer cover of $Y$. Then,
\begin{itemize}
\item[(b)] $\text{per} (Y) = \text{per} (X)$.
\end{itemize}
If $Z$ is an irreducible sofic shift and $Y$ is a factor of $Z$, then	
\begin{itemize}
\item[(c)] 
$\text{per}(Y)$ divides $\text{per}(Z)$.
\end{itemize}
%\begin{itemize}
%\item[(c)]	$\lim_{n \to \infty} \frac{1}{n\text{per} (Y)} \log r_{n\text{per} (Y)}(\pi) = h(Y)$.
%\end{itemize}	
\end{lemma}

\begin{proof} (b):  Let $p_1,p_2, \cdots, p_K$ be  receptive periodic points in $Y$ such that $\text{per}(Y) = \gcd\{\per(p_i): \ i = 1,2,\cdots, K\}$. 
%, where we denote the least period of a periodic point $p$ by $\Per(p)$. 
From Theorem \ref{16-03-23a} and $K-1$ applications of Lemma \ref{11-03-23b} we get an irreducible SFT cover $\rho : Z \to Y$ with periodic points $z_1,z_2, \cdots, z_K \in Z$ such that $\per(z_i) = \per(p_i)$ for all $i$. Since $\text{per}(Z)$ divides $\per(z_i)$ for all $i$, it follows that $\text{per}(Z)$ divides $\text{per}(Y)$. By Corollary \ref{05-07-25d} this implies that $\text{per}(X)$ divides $\text{per}(Y)$. To obtain (b) we show that, conversely, $\text{per}(Y)$ divides $\text{per}(X)$. Let $(G,\mathcal L)$ be the labeled graph such that $X = X_G$ and $\pi = \mathcal L_\infty$. %If $G$ is a single {\color{blue}{cycle}}, $\pi$ is a conjugacy and (b) is trivial. We can therefore assume that this is not the case. 
Let $p$ be an arbitrary cycle in $G$ s.t. $|p|= \per(p^\infty)$. Let $r$ be a cycle in $G$ whose label contains a synchronizing word. We may assume that
$i_G(r) = t_G(r) = i_G(p) = t_G(p)$.  Let $p_1:=(pr)^\infty$  and $p_2:=(p^2r)^\infty$. Then $p_1$ and $p_2$ are periodic points which contain synchronizing words and therefore $\pi(p_1)$ and $\pi(p_2)$ are receptive. 
Since $|pr|$ and $|p^2r|$ are periods (not necessarily least periods) of $p_1$ and $p_2$ respectively, 
 $per(Y)$ divides $|pr|$ and $|p^2r|$ and therefore divides $|p|$.
%
%Using Proposition 3.3.16 of \cite{LM} we can construct a {\color{blue}{cycle}} $r$ in $G$ whose label is a magic word for $Y$ such that $s_G(r) = t_G(r) = s_G(p) = t_G(p)$. Since $G$ is not just the {\color{blue}{cycle}} $p$ we can also arrange that $r \nsubseteq p^\infty$. Then are periodic points in $X_G$ and since $r$ labels a magic word it follows that $\per(\pi(p_i)) = \per(p_i), i = 1,2$. Then $\pi(p_i)$ is receptive by Corollary \ref{16-03-23a}, and
%%$$
%%\per (\pi(p_2)) - \per(\pi(p_1)) =  \per (p_2) - \per(p_1) =|p|.
%$$
Since $p$ was arbitrary, we have that $\text{per}(Y)$ divides $\text{per}(X)$, as claimed.

(a): This follows now from (b) and Corollary \ref{05-07-25d}.

(c): Let $\psi: Z\to Y$ be a factor code and $\phi: X\to Z$ be the Fischer cover of $Z$. Noting that $\psi\circ \phi: X\to Y$ is a factor code, we infer from (a) that $\p(Y)$ divides $\p(X)$. Moreover, by (b) we have  $\p(Z)=\p(X)$. Thus, $\p(Y)$ divides $\p(Z)$.
\end{proof}

%When the Fischer cover $X$ of $Y$ is mixing, which happens if %and only if $Y$ is mixing, the equality in (c) of Lemma \ref{03-06-25} follows from Proposition 4.3 of \cite{M}. In fact, MacDonald proves that the equality holds for all mixing covers of $Y$; not only the Fischer cover.

\begin{lemma} \label{growthrate_rec}
	Let $Y$ be an irreducible sofic shift and $\pi: X\to Y$ be the Fischer cover of Y. Then	
		$$
		\lim_{n \to \infty} \frac{1}{n\text{per} (Y)} \log r_{n\text{per} (Y)}(\pi) = h(Y).
		$$
\end{lemma}	
\begin{proof}
	 Note first of all that $r_j(\pi) \leq q_j(X)$ for all $j$, almost by definition. Since $\text{per}(Y) = \text{per}(X)$ by Lemma~\ref{03-06-25}(b),  we get
	\begin{equation}\label{05-06-25dx}
		\limsup_{n \to \infty} \frac{1}{n\text{per} (Y)} \log r_{n\text{per} (Y)}(\pi) \leq \limsup_{j\to \infty} \frac{1}{j} \log q_{j}(X) = h(X)
	\end{equation}
	by \cite[Corollary 4.3.8]{LM}. In particular, Lemma \ref{growthrate_rec} holds when $h(Y) = h(X) = 0$, and we assume therefore henceforth that $h(X) > 0$. Note that $X = X_G$ and $\pi = \mathcal L_\infty$, where $\mathcal G =(G ,\mathcal L)$ is irreducible, right-resolving and follower-separated. It follows from \cite[Proposition 3.3.16]{LM} that there is a cycle $\gamma$ in $G$ such that $m= \mathcal L(\gamma)$ is magic for $\mathcal G$. Thus $m$ is also synchronizing for $Y = X_\mathcal G$ by Lemma \ref{19-03-23d}. Let $v$ be the vertex in $G$ where $\gamma$ starts and terminates. Let $Z_n$ be the subset of $X_G$ consisting of the bi-infinite paths in $G$ of the form
	$$
	\cdots b_{-3}\gamma b_{-2} \gamma b_{-1} \gamma b_0 \gamma b_1 \gamma b_2 \gamma b_3 \cdots
	$$
	where $|b_i | \leq n$ for all $i \in \mathbb Z$. Then
	\begin{itemize}
		\item[(i)] $Z_n$ is an irreducible SFT, and
		\item[(ii)] $\# \pi^{-1}(\pi(y)) = 1$ for all $y \in Z_n$.
	\end{itemize}
	We leave the reader to verify that there is an $N\in \mathbb N$ such that for all $n\geq N$ we have also that
	\begin{itemize}
		\item[(iii)] $\text{per}(Z_n) = \text{per}(X)$.
	\end{itemize}
	Let $\epsilon > 0$. Let $A_n$ denote the set of cycles $a$ in $G$ of length $|a| =n$ and with ${i_G(a)} = t_G(a) = v$. As is well-known,
	$$
	\limsup_{k \to \infty} \frac{1}{k} \log (\# A_k) = h(X) ,
	$$
	and we can therefore choose $n \geq N$ so large that $\frac{n}{n+|\gamma|}\geq 1-\epsilon$ and
	$$
	\# A_n \geq e^{n(h(X) - \epsilon)}.
	$$
	Then $\mathcal B_{k(n+|\gamma|)}(Z_n) \geq (\# A_n)^k$ and hence
	\begin{equation}\label{08-07-25}
		\begin{split}
			&h(Z_n) = \limsup_{k \to \infty} \frac{1}{k} \log \left( \# \mathcal B_k(Z_n)\right)  \\
			&\geq \limsup_{k \to \infty} \frac{1}{k(n+|\gamma|)} \log \left((\# A_n)^k\right)\\
			& = \frac{n}{n+|\gamma|} \frac{1}{n}\log (\# A_n)
			\geq (1-\epsilon)(h(X) - \epsilon) .
		\end{split}
	\end{equation}
	Thanks to (iii) we have also that
	\begin{equation}\label{08-07-25a}
		\lim_{k \to \infty} \frac{1}{k \text{per}(X)} \log q_{k \text{per}(X)}(Z_n) = h(Z_n).
	\end{equation}
	Note that it follows from (ii) that $\pi$ restricts to an injective map
	$$
	\pi : \ Q_k(Z_n) \to Q_k(Y) \cap \pi(Q_k(X)) ,
	$$
	implying that $q_k(Z_n) \leq r_k(\pi)$ for all $k \in \mathbb N$. Combined with \eqref{08-07-25} and \eqref{08-07-25a} this gives the estimate
	\begin{equation}\label{09-07-25}
		\liminf_{k \to \infty} \frac{1}{k \text{per}(X)} r_{k \text{per}(X)}(\pi) \geq (1-\epsilon)(h(X) - \epsilon) .
	\end{equation}
	Since $\text{per}(X) = \text{per}(Y)$ by Lemma~\ref{03-06-25}(b), the result follows by combining \eqref{05-06-25dx} and \eqref{09-07-25}.
\end{proof}	

When the Fischer cover $X$ of $Y$ is mixing, which happens if and only if $Y$ is mixing, the equality in Lemma \ref{growthrate_rec} follows from Proposition 4.3 of \cite{M}. In fact, MacDonald proves that the equality holds for all mixing covers of $Y$; not only the Fischer cover.
%\begin{cor}
%\label{divides}
%If $Y, Z$ are irreducible sofic shifts and $Y$ is a factor of $Z$, then $\p(Y)$ divides $\p(Z)$.
%\end{cor}
%\begin{proof}
%Let $\pi: Z\to Y$ be a factor code and $\phi: X\to Z$ be the Fischer cover of $Z$. Noting that $\pi\circ \phi: X\to Y$ is a factor code, the conclusion follows from Corollary \ref{05-07-25d} and (b) of Lemma \ref{03-06-25}.
%\end{proof}

%{\color{blue}{
%\begin{cor}
%\label{divides}
%If $Y, Z$ are irreducible sofic shifts and $Y$ is a factor $Z$, then $\p(Y)$ divides $\p(Z)$.
%\end{cor}
%\begin{proof}
%Let $\pi: Z\to Y$ be a factor code and $\phi: X\to Z$ be the Fischer cover of $Z$. Noting that $\pi\circ \phi: X\to Y$ is a factor code, we infer from Lemma~\ref{03-06-25}(a) that $\p(Y)$ divides $\p(X)$. By Lemma~\ref{03-06-25}(b), $\p(Z) = \p(X)$.  Thus $\p(Y)$ divides $\p(Z)$.
%\end{proof}
%}}

\subsection{The main result on factorizable embeddings}
\label{mainfacemb}

We quickly extract from \cite{T} some definitions, terminology and results that we will need in the following.

Let $X$ be a shift space, $R(X)$ be the closure of the set of periodic points of $X$ and $S(X)$ be the set of synchronizing words of $R(X)$. For $s, t\in S(X)$, we write $s\sim t$ when there exist $x, y\in \mathcal{B}(R(X))$ such that $sxt, tys\in \mathcal{B}(R(X))$. This is an equivalence relation and we let $\mathcal{S}(X) =S(X)/ \sim$ denote its equivalence classes.

For each $\alpha\in \mathcal{S}(X)$, the {\em irreducible component at level $0$}  with respect to $\alpha$, denoted $X_{(\alpha, 0)}$, is  the set of points $x\in R(X)$ such that elements of $\alpha$ appear with bounded gap in $x$ (the bound may depend on the point $x$).
%$$
%\sup (\inf \{j>i: \mbox{there is a synchonizing word $w$ such that %$w\subseteq x_{[i,j[}$}\})
%$$
%is finite.

The {\em derived shift} of $X$, denoted $\partial X$,  is the set of points in $R(X)$  which do not contain any synchronizing word of $R(X)$. It is not hard to see that $\partial X$ is a shift space.  So, we can iterate this operation to get a nested sequence of shift spaces: with $\partial^0 X:= X$, the {\em $k$-th derived shift} is defined $\partial^k X:= \partial(\partial^{k-1} X)$. The {\em depth of $X$} is the unique integer $k$ (if any) such that $\partial^{k} X \ne \emptyset$ and $\partial^{k+1} X = \emptyset.$

An {\em irreducible component of $X$ at level $k$} is an irreducible component at level $0$ of $\partial^k X$. An {\em irreducible component of $X$} is an irreducible component of $X$ at some level.  It is not hard to see that the irreducible components of $X$ are disjoint.

If $X$ is sofic, then $\partial X$ is also sofic (see \cite[Theorem 6.6]{T}) and has finitely many irreducible components of level 0. Also, $X$ has finite depth (see \cite[Theorem 6.7]{T}). As a consequence, $X$ has finitely many irreducible components (see \cite[Proposition 6.18]{T}).  Moreover, the closure of an irreducible component of a sofic shift is an irreducible sofic shift (see \cite[Proposition 6.16]{T}). Finally, we note that since a sofic shift has finite depth, each periodic point is contained in a unique irreducible component, (see \cite[Lemma 4.2]{T}).

In the following lemma, we set out some important consequences of the material above.

\begin{lemma} \label{lemma_x}
Let $U \subseteq V$ be sofic shift spaces with $U$ irreducible. Then,

\begin{enumerate}
	\item[(1)]
	There is a unique irreducible component $Y_c$ of $V$  such that
	\begin{itemize}
	\item[(a)] $U \subseteq \overline{Y_c}$, and
	\item[(b)] $U \cap Y_c \ne  \emptyset$.
	\end{itemize}
\end{enumerate}
In addition,
\begin{enumerate}	
	\item[(2)] For all $n$, $\Rec_n(U) \subseteq \Rec_n(\overline{Y_c})$, and
	\item[(3)]
	if $W$ is an irreducible SFT and $\pi:W \to U$ is a factor code, then $W$ is $\p(\overline{Y_c})$-periodic.
\end{enumerate}
\end{lemma}
\begin{proof}
(1). Assume first that $Y_c$ is an irreducible component in $Y$ such that (a) and (b) hold. Then $U \subseteq \partial^k V$, where $k$ is the level of $Y_c$; in fact, $U \subseteq R(\partial^kV)$ since $U$ is irreducible. Because of (b), $U$ contains a synchronizing word for $R(\partial^kV)$ and $U$ is therefore not contained in $\partial^{k+1}V$. It suffices therefore to show that when $k$ is the largest integer for which $U \subseteq \partial^k V$, there is a unique irreducible component at level $k$ for which (a) and (b) hold. For this note that $U \subseteq R(\partial^kV)$ and that $U$ must contain a synchronizing word $w$ for  $R(\partial^kV)$ since $U \nsubseteq \partial^{k+1}V$. Let $Y_c$ be the unique irreducible component at level $k$ whose language contains $w$, and let $u$ be an arbitrary word in the language of $U$. Since $U$ is irreducible and sofic, there are words $s,t$ in $U$ such that $U$ contains the periodic point $(wsut)^\infty$. This periodic point is contained in some component of $Y$, and since it lies in $R(\partial^kV)$ and contains $w$, this component must be $Y_c$. Thus $(wsut)^\infty \in U \cap Y_c \neq \emptyset$ and hence (b) holds. Since $u$ was an arbitrary word in $U$, so does (a). For the uniqueness, assume that $Y_{c'}$ is a component at level $k$ such that $U \cap Y_{c'} \ne  \emptyset$. Then $U \cap Y_{c'}$ contains a word $w'$ which is synchronizing for $R(\partial^k V)$. Since $U$ is irreducible this implies that $w$ and $w'$ are equivalent in $S(\partial^k V)$, and hence that $Y_{c'} = Y_c$.

(2). Let $x \in \Rec_n(U)$, with synchronizing words $u,v$ in $U$ s.t. for all $i$, $ux_{[0,n-1]}^iv$ is in the language of $U$. Let $w$ be a synchronizing word in $\overline{Y_c}$. By irreducibility of $\overline{Y_c}$, there are words $s,t$ s.t. $wsu, vtw$ are in the language of $\overline{Y_c}$. Then $wsux_{[0,n-1]}^ivtw$ is in the language of $\overline{Y_c}$ for all $i$. Since $w$ is synchronizing in $\overline{Y_c}$, $x \in \Rec_n(\overline{Y_c})$.

(3).  By (2), $\p(\overline{Y_c})$ divides $\p(U)$, and $\p(U)$ divides $\p(W)$ by (a) of Lemma \ref{03-06-25}. Hence $\p(\overline{Y_c})$ divides $\p(W)$, and therefore $W$ is $\p(\overline{Y_c})$-periodic.
\end{proof}

%As mentioned above we shall now apply notions and results from \cite{T}; in particular, the notion of irreducible components of sofic shifts will play an important role. We refer to \cite{T} for this as well as other things that will be needed.

The following lemma, which is based on Lemma~\ref{lemma_x} above, is key, to the results in this section. It says, among other things, that if a subshift $Z$ factorizably embeds in an irreducible sofic shift $Y$, then $Z$ must actually embed in the closure of exactly one of the (finitely many) irreducible components $Y_c$ of $Y$, in such a way that all of the periodic points of $Z$ map to receptive periodic points of
$\overline{Y_c}$.

\begin{lemma}\label{20-03-23d} Let $Z$ be a subshift and $\phi : Z \to Y$ a factorizable embedding, where $Y$ is an irreducible sofic shift. There is an irreducible component $Y_c$ of $Y$ such that $\phi(Z) \subseteq \overline{Y_c}$, $Z$ is $\text{per}(\overline{Y_c})$-periodic and $\phi(Q_n(Z)) \subseteq \Rec_n(\overline{Y_c})$ for all $n \in \mathbb N$.
\end{lemma}
\begin{proof}
Consider the diagram \eqref{20-03-23c}, with the irreducible SFT $S$. Since $\phi$ and $\psi$ are injective we have
$$
\psi(Q_n(Z)) \subseteq Q_n(S) \ \text{and} \ \pi(\psi(Q_n(Z))) = \phi(Q_n(Z)) \subseteq Q_n(\pi(S)).
$$
It follows therefore from Theorem \ref{16-03-23a} that $\phi(Q_n(Z)) \subseteq \Rec_n(\pi(S))$ for all $n \in \mathbb N$.

Applying Lemma~\ref{lemma_x} part (1) with $U=\pi(S)$ and $V=Y$, we know that there is a unique irreducible component $Y_c$ of $Y$ such that $\pi(S) \subseteq \overline{Y_c}$ and $\pi(S) \cap Y_c \neq \emptyset$. Lemma \ref{lemma_x} part (2) (with $U=\pi(S)$ and $V=Y$) gives that
$\Rec_n(\pi(S)) \subseteq \Rec_n(\overline{Y_c})$. Hence $\phi(Q_n(Z)) \subseteq  \Rec_n(\overline{Y_c})$ for all $n$.

To see that $Z$ is $\text{per}(\overline{Y_c})$-periodic,
%note that $\text{per}(\overline{Y_c})$ divides $\text{per}(\pi(S))$ since $\Rec_n(\pi(S)) \subseteq \Rec_n(\overline{Y_c})$. It follows from a) of Lemma \ref{03-06-25} above that $\text{per} (\pi(S))$ divides $\text{per}(S)$.
note that $S$ is $\p(\overline{Y_c})$-periodic by applying part (3) of Lemma \ref{lemma_x} with $W=S, U=\pi(S)$ and $V=Y$. Hence $\text{per}(S) = k \text{per}(\overline{Y_c})$ for some $k \in \mathbb N$. Since $Z$ embeds into $S$ (via $\psi$), $Z$ must be $k \text{per}(\overline{Y_c})$-periodic and hence also $\text{per} (\overline{Y_c})$-periodic.
\end{proof}

\begin{lemma}\label{20-03-23e}  Let $Z$ be a subshift and $\phi : Z \to Y$ an $S$-factorizable embedding, where $Y$ is an irreducible sofic shift. Then $Z$ is $\p(Y)$-periodic and $\phi(Q_n(Z)) \subseteq \Rec_n(Y)$ for all $n \in \mathbb N$.
\end{lemma}
\begin{proof} This follows from Lemma \ref{20-03-23d} and its proof: When $\pi$ is surjective, $\overline{Y_c} = \pi(S) = Y$.
\end{proof}

\begin{lemma}\label{17-06-24b} Let $Z$ be a subshift and $Y$ an irreducible sofic subshift. Assume that $h(Z) < h(Y)$. There is an $S$-factorisable embedding of $Z$ into $Y$ if and only if
\begin{itemize}
\item[(i)] $Z$ is $\text{per} (Y)$-periodic, and
\item[(ii)] $q_{n \text{per}(Y)}(Z) \leq \rec_{n \text{per}(Y)}(Y)$ for all $n \in \mathbb N$.
\end{itemize}
\end{lemma}
\begin{proof} The necessity of the two conditions, (i) and (ii), follows from Lemma \ref{20-03-23e}.

For the converse, assume that (i) and (ii) hold.  Let $\pi : X \to Y$ be the Fischer cover of $Y$. By (b) of Lemma \ref{03-06-25}, $\text{per} (X) = \text{per}(Y)$. Note that
$$
\limsup_n \frac{1}{n  \text{per}(Y)} \log q_{n \text{per} (Y)} (Z)\leq \limsup_n \frac{1}{n} \log q_n(Z) \leq h(Z)
$$
while $\lim_{n \to \infty} \frac{1}{n \text{per}(Y)} \log r_{n  \text{per}(Y)}(\pi) = h(Y)$ by Lemma \ref{growthrate_rec}. Since $h(Z) < h(Y)$ by assumption, it follows from this and (ii) that there is an $N\in \mathbb N$ such that $q_{n\text{per}(Y)}(Z) \leq  r_{n  \text{per}(Y)}(\pi)$ for all $n \geq N$. Since we assume (ii) we can apply Lemma \ref{11-03-23b} a finite number of times to get an irreducible SFT cover $\pi': X' \to Y$ of $Y$ with $\text{per}(X') =\text{per} (Y)$ such that $q_{n \text{per}(Y)}(\pi') \leq r_{ n \text{per}(Y)}(\pi')$ for all $n \in \mathbb N$. It follows then from Theorem \ref{Sophie3} that $Z$ embeds into $Y$ via $\pi'$.

\end{proof}

\begin{lemma}\label{17-06-24l} Let $Z$ and $Y$ be subshifts; $Y$ irreducible sofic. Assume that $h(Z) < h(Y)$. Consider the following conditions.
\begin{enumerate}
\item $Z$ is $\text{per}(Y)$-periodic and $q_{ n\text{per}(Y)}(Z) \leq \rec_{n\text{per}(Y)}(Y)$ for all $n \in \mathbb N$.
\item There is an S-factorizable embedding $Z \hookrightarrow Y$.
\item There is an I-factorizable embedding $Z \hookrightarrow Y$.
\item  There is a factorizable embedding $Z \hookrightarrow Y$.
\end{enumerate}
Then (1) $\Leftrightarrow$ (2) $\Rightarrow$ (3) $\Rightarrow$ (4) and when $h(\partial Y) < h(Z) < h(Y)$ all four conditions are equivalent. ($\partial Y$ is the derived shift space of $Y$)%, cf. \cite{T}.)
\end{lemma}
\begin{proof} (1) $\Leftrightarrow$ (2) follows from Lemma \ref{17-06-24b}. The implication (2) $\Rightarrow$ (3) was pointed out in Remark \ref{18-06-24e}, and the implication (3) $\Rightarrow$ (4) is obvious.  Assume now that (4) holds and $h(\partial Y) < h(Z)$. By \cite[Lemma 3.5]{T} and the definition of the derived shift, all components of $Y$, in the sense of \cite{T}, other than the top component are contained in $\partial Y$. Since we assume $h(Z) > h(\partial Y)$, the component $Y_c$ occurring in Lemma~\ref{20-03-23d} can only be the top component whose closure is all of $Y$ by \cite[Lemma 3.5]{T} again. It follows then from Lemma~\ref{20-03-23d}
 that $Z$ is $\text{per} (Y)$-periodic and that $q_{n\text{per}(Y)}(Z) \leq \rec_{n\text{per} (Y)}(Y)$ for all $n \in \mathbb N$. Hence (4) $\Rightarrow$ (1) when $h(Z) < h(\partial Y)$.
\end{proof}

\begin{thm}\label{18-06-24} Let $Z$ and $Y$ be subshifts, $Y$ sofic and irreducible. The following are equivalent:
\begin{itemize}
\item[(a)] There is a factorisable embedding $Z \hookrightarrow Y$.
\item[(b)] There is an $I$- factorisable embedding $Z \hookrightarrow Y$.
\item[(c)] There is an irreducible component $Y_c$ in $Y$ such that either
\begin{itemize}
\item[$\cdot$] $Z$ and $\overline{Y_c}$ are conjugate SFTs, or
\item[$\cdot$] $Z$ is $\p (\overline{Y_c})$-periodic, $q_{n \p (\overline{Y_c})}(Z) \leq  \rec_{n\p (\overline{Y_c})}(\overline{Y_c})$ for all $n$ and $h(Z) < h(\overline{Y_c})$.
\end{itemize}
\item[(d)] There is an irreducible component $Y_c$ in $Y$ and an S-factorizable embedding $Z \hookrightarrow \overline{Y_c}$.
\end{itemize}
 \end{thm}
 \begin{proof}  (a) $\Rightarrow$ (c):  Let $\phi :Z \to Y$ be a factorizable embedding such that we have the diagram \eqref{20-03-23c}. It follows from Lemma \ref{20-03-23d} and its proof that there is an irreducible component $Y_c$ of $Y$ such that $Z$ is
	$\p (\overline{Y_c})$-periodic, $\phi(Z) \subseteq \pi(S) \subseteq \overline{Y_c}$ and $q_{n\p(\overline{Y_c})}(Z)
	\leq \rec_{n\p (\overline{Y_c})}(Y)$ for all $n \in \mathbb N$. If $h(Z)<h(\overline{Y_c})$, then we have the second bullet of (c); otherwise, if $h(Z)=h(\overline{Y_c})$, then we have the first bullet of (c), by the implication of $(1)\Rightarrow (3)$ of Theorem \ref{Sophie3}, where we recall that  $\overline{Y_c}$ is sofic and irreducible by Proposition 6.16 in \cite{T}. %Therefore (c) follows from Theorem \ref{Sophie3}.

  (c) $\Rightarrow$ (d): If the first item of (c) holds, then so does (d), trivially. That (d) is also true when the second item of (c) holds follows from Lemma \ref{17-06-24b}.

 (d) $\Rightarrow$ (b): First recall that $\overline{Y}_c$ is irreducible and sofic. Then, it follows from the direction (a) $\Rightarrow$ (c), applied to the assumed factorizable embedding $Z \hookrightarrow \overline{Y_c}$, that there is an irreducible component $Y_{cc}$ of $\overline{Y_c}$ such that either $Z$ and $\overline{Y_{cc}}$ are conjugate SFTs, or the second item of (c) holds with $Y_c$ replaced by $Y_{cc}$. In the former case, (b) trivially holds; in the latter case, it follows from Lemma \ref{17-06-24b} and Remark \ref{18-06-24e} that there is an $I$-factorizable embedding from $Z$ to $\overline{Y_{cc}}$, which in particular is also an $I$-factorisable embedding from $Z$ to $Y$, proving (b).

 (b) $\Rightarrow$ (a) is trivial and hence (a) through (d) are equivalent.

 \end{proof}

When assuming $h(Z)> h(\partial Y)$, we immediately have the following corollary.

\begin{cor}\label{18-06-24f} Let $Z$ be a subshift and $Y$ an irreducible sofic shift such that $h(Z) > h(\partial Y)$. Set $p:= \p (Y)$. Then, the following are equivalent:
	\begin{itemize}
		\item[(a)] There is a factorizable embedding $Z \hookrightarrow Y$.
		\item[(b)] There is an $I$- factorisable embedding $Z \hookrightarrow Y$.
		\item[(c)] Either
		\begin{itemize}
			\item[$\cdot$] $Z$ and $Y$ are conjugate SFTs, or
			\item[$\cdot$] $Z$ is $p$-periodic, $q_{n p}(Z) \leq  \rec_{n p}(Y)$ for all $n$ and $h(Z) < h(Y)$.
		\end{itemize}
		\item[(d)] There is an S-factorizable embedding $Z \hookrightarrow Y$.
		
	\end{itemize}
\end{cor}
\begin{proof} When $h(Z) > h(\partial Y)$, the entropy of $\overline{Y_c}$ will be strictly smaller than $h(Z)$ for all other irreducible components $Y_c$ than the top component in $Y$. Therefore, if either (c) or (d) in Theorem \ref{18-06-24} holds, it follows that $ \overline{Y_c} = Y$. Thus, the corollary follows from Theorem \ref{18-06-24}.
\end{proof}

In general, the conditions in Theorem \ref{18-06-24} will not imply the existence of an $S$-factorizable embedding, but nonetheless we can give the following characterization of when such an embedding exists.

\begin{thm}\label{05-06-25e} Let $Z$ and $Y$ be subshifts, $Y$ sofic and irreducible. The following conditions (a) and (b) are equivalent:
\begin{itemize}
\item[(a)]  There is an $S$-factorisable embedding $Z \hookrightarrow Y$.
\item[(b)]
\begin{itemize}
\item[$\cdot$] $Z$ and $Y$ are conjugate SFTs, or
\item[$\cdot$] $Z$ is $\text{per} (Y)$-periodic, $q_{n \text{per}(Y)}(Z) \leq  \rec_{n\text{per}(Y)}(Y)$ for all $n$ and $h(Z) < h(Y)$.
\end{itemize}
\end{itemize}
\end{thm}
\begin{proof} $(b) \Rightarrow (a)$ follows as in the proof of $(c) \Rightarrow (d)$ in Theorem \ref{18-06-24} with $Y$ in the place of $Y_c$.  $(a) \Rightarrow (b)$: If $h(Z) < h(Y)$ it follows from Lemma \ref{17-06-24b} that the second item in (b) holds. If $h(Z) = h(Y)$, we are in the situation covered by Corollary \ref{18-06-24f} because $h(Y) > h(\partial Y)$. It follows therefore from Corollary \ref{18-06-24f} that $Z$ and $Y$ are conjugate SFTs.

\end{proof}

We conclude this section with the following two examples, showing that factorizable embedding is a (strictly) stronger notion than embedding, and $S$-factorizable embedding is (strictly) stronger than factorizable embedding.

\begin{example}\label{17-02-25}
	Let $Z$ be the intersection of the even shift and the golden mean shift. See Figure \eqref{20-02-25b} for a presentation of $Z$; it's the factor of the edge shift defined by the labeling. Let $Y$ be the even shift.

	
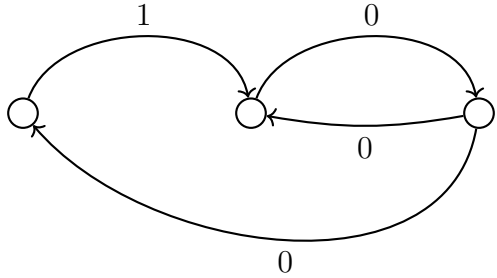
\begin{equation}\label{20-02-25b}
	\begin{tikzpicture}[node distance={30mm}, thick, main/.style = {draw, circle}]
\node[main] (1) {};
\node[main] (2) [right of=1] {};
\node[main] (3) [right of=2] {};
\draw[->] (1) to [out=70,in=100,looseness=1] node[above,pos=0.5] {$1$} (2);
\draw[->] (2) to [out=70,in=100,looseness=1] node[above,pos=0.5] {$0$} (3);
\draw[->] (3) to [out=190,in=350,looseness=1] node[below,pos=0.5] {$0$} (2);
\draw[->] (3) to [out=260,in=310,looseness=1] node[below,pos=0.5] {$0$} (1);
 \end{tikzpicture}
 \end{equation}

    Note that $Z$ and $Y$ are both mixing sofic and that $Z$ embeds into $Y$. Now we claim that there is no factorizable embedding from $Z$ to $Y$. By Theorem \ref{18-06-24} ((a) $\iff$ (b)), it suffices to show there is no $I$-factorizable embedding from $Z$ into $Y$. To this end, suppose to the contrary that such an $I$-factorizable embedding exists. Since $Z$ is a subshift of $Y$, there is then an irreducible SFT $U$ such that $Z \subseteq U \subseteq Y$. Thus, for all $n$, $10^n$ and $0^n1$ must be allowed in $U$ and therefore $10^n1$ is allowed in $U$ for all large $n$ (i.e., $n$ greater than the memory of $U$). But this contradicts the fact that $U$ is contained in the even shift. %Since $Z$ is irreducible, $U$ must contain an irreducible component $U'$ such that $Z\subset U'\subset Y$.

\end{example}

\begin{example}

	Let $X=\{0^\infty\}$ and let $Z$ be the same sofic shift as in Example \ref{17-02-25}, the intersection of the even shift and the golden mean shift. Since $0^\infty\in Z$, $X$ clearly embeds factorizably in $Z$. However, noting that $0^\infty$ is the only fixed point in $Z$ and it is not receptive, we have $r_1(Z)=0$ and therefore $1=q_1(X)>r_1(Z)=0$. By Lemma \ref{17-06-24b}, there is no $S$-factorizable embedding from $X$ into $Z$.

\end{example}

\section{Period of irreducible sofic shifts}\label{period}

In this section we investigate the notion of period of an irreducible sofic shift $Y$, which is a crucial part of two of our main results, Theorem \ref{18-06-24} and Theorem \ref{05-06-25e}. We defined the period of $Y$ in Definition~\ref{15-06-25}, as the gcd of periods of receptive periodic points in $Y$ and showed, in Lemma \ref{03-06-25}, that this is the same as the global period of the right Fischer cover of $Y$. A broad generalization, given in the following result, is that eight plausible definitions of the period of $Y$ are all equivalent.

\begin{prop} \label{25-03-23-b}
	Let $Y$ be an irreducible sofic shift. Define:
	\begin{enumerate}
\item[$p_1$:] The global period of any irreducible SFT $X$ such that there is a factor code $X \to Y$ for which some point in $Y$ has a unique pre-image in $X$.
    \item[$p_2$:] The global period of any irreducible SFT which factors onto $Y$ under an almost invertible code.
     \item[$p_3$:]  The global period of the right or left Fischer cover of $Y$.
		\item[$p_4$:] The gcd of least periods of receptive periodic points of $Y$.
		\item[$p_5$:] The gcd of least periods of periodic points of $Y$ which contain a synchronizing word.
        		%\item[$p_4$:] The gcd of least periods of periodic points in $Y_c$, where $Y_c$ is the top component of $Y$.
		\item[$p_6$:] The unique number $m$ in a collection $D_0,D_1,D_2, \cdots, D_{m-1}$ of distinct closed subsets of $Y$ with the properties that
        \begin{itemize}
            \item[(i)] $Y = \bigcup_{i=0}^{m-1} D_i$,
            \item[(ii)] $\sigma(D_i) = D_{i+1 \bmod m}$, and
            \item[(iii)] $\sigma^m|_{D_i}$ is mixing for one, and hence for all $i$.
            \item[(iv)] $D_i\cap D_j$ has empty interior when $i,j\in \{0,1,\cdots, m-1\}$ and $i\neq j$.
        \end{itemize}
       \item[$p_7$:] The unique number $n$ in a collection $C_0, C_1, C_2, \cdots, C_{n-1}$ of distinct relatively clopen subsets of $\tilde{Y}$, where $\tilde{Y}$ is the set of doubly transitive points, such that
       \begin{itemize}
       \item[(a)] $\tilde{Y}=\bigcup_{i=0}^{n-1} C_i$;
       \item[(b)] $\sigma(C_i)=C_{i+1 \bmod n}$;
       \item[(c)] $\sigma^n$ is mixing on each $C_i$, in the sense that for any two relatively open sets $U, V$ in $C_i$, $(\sigma^n)^k(U)\cap V \neq \emptyset$ for all sufficiently large $k$,
       \item[(d)] $C_i$'s are mutually disjoint.
       \end{itemize}
       \item[$p_8$:]  The minimum of global periods of irreducible SFTs which factor onto $Y$, i.e.,
\begin{align*}
p_8 =\min \bigl\{\p(X): X \ \text{is an irreducible SFT} \\ \text{that factors onto} \ Y\bigr\}.
\end{align*}
\end{enumerate}
    Then $p_1=p_2=p_3=p_4=p_5=p_6=p_7=p_8$.
\end{prop}
\begin{proof}

	$(p_1=p_2=p_3=p_8)$: It follows from Lemma \ref{05-07-25} that the minimum occurring in the definition of $p_8$ is realized by any irreducible SFT cover of $Y$ for which there is a point in $Y$ with a unique pre-image under the corresponding factor code. Such points in $Y$ exist in abundance when the cover is almost invertible, in particular the right or left Fischer cover,  and hence $p_1 = p_2 = p_3= p_8$.

    ($p_5=p_3$): Let $\pi:X \to Y$ be the right Fischer cover.  Then $X = X_H$ for some irreducible graph $H$, $\pi$ is given by a labeling $\mathcal{L}$ of the edges of $H$ and $p_3 = \p(X_H)$.

    Let $\gamma$ be a cycle in $H$. To show $p_5$ divides $\p(X_H)$, it suffices to show that $p_5$ divides $|\gamma|$. %(since $\p(Y)$ clearly devides $p_1$, which is equivalent to $p_4$).
	To this end, let $u$ be a synchronizing word in $Y$ and let $\xi\in\mathcal{L}^{-1}(u)$, which is a path in $H$. By the irreducibility of $X_H$, there exist paths $x, y$ in $G$ such that $i_H(x)=t_H(\xi), t_H(x)=i_H(\gamma), i_H(y)=i_H(\gamma), t_H(y)=i_H(\xi)$. Now, let $\alpha:=\xi x y$ and $\beta=\xi x \gamma y$.  Both $\mathcal{L}(\alpha^\infty)$ and $\mathcal{L}(\beta^\infty)$ are periodic points in $Y$ which contain the synchronizing word $u$.  Thus, $p_5$ divides $|\alpha|$ and $|\beta|$, and therefore divides their difference $|\gamma|$, as desired.

    It remains to show that $\p(X_H)$ divides $p_5$. This follows from the fact that every periodic point of $Y$ that contains a synchronizing word has a unique pre-image in the right Fischer cover, which is a consequence of Lemma \ref{19-03-23d}.

$(p_4=p_3)$: This is (b) of Lemma \ref{03-06-25}.

$(p_5=p_6):$ It is easy to see that $p_5$ is the gcd of the least periods of periodic points in the unique irreducible component at level $0$ of $Y$. Then, $p_5=p_6$ follows from Corollary 3.12 and Lemma 3.13 in \cite{T}.

$(p_6=p_7):$
We first show that the decomposition $C_0, C_1, \cdots, C_{n-1}$ exists with $n=p_6$. Let $D_0,D_1,\cdots, D_{p_6-1}$ be the sets from the definition of $p_6$. For each $0\leq i\leq p_6-1$, let $U_i$ be the interior of $D_i$. It follows from (i) and the Baire category theorem that $U_i$ is non-empty for at least one $i$, and then from (ii) that $U_i \neq \emptyset$ for all $i$. Set $C_i := U_i \cap \tilde{Y}$.  The conditions (b) and (d) are clear. It follows from (ii) that $\cup_{i=0}^{p_6-1} U_i$ is $\sigma$-invariant. As the union is also open it follows that $\cup_{i=0}^{p_6-1}U_i$ must contain all doubly transitive points and therefore condition (a) holds. For part (c), first note that there exist open sets $U', V'\subseteq U_i$ such that $U=U'\cap \tilde{Y}, V=V'\cap \tilde{Y}$. Since $(D_i, \sigma^{p_6})$ is mixing by (iii), for all sufficiently large $k$, $\sigma^{kp_6}(U') \cap V' \neq \emptyset$. Now,
     \begin{align*}
         \sigma^{kp_6}(U) \cap V = \sigma^{kp_6}(U'\cap \tilde{Y}) \cap (V'\cap \tilde{Y}) = \sigma^{kp_6} (U') \cap V' \cap \tilde{Y} \neq \emptyset
     \end{align*}
     where the last step follows from the fact that $\tilde{Y}$ is dense in $Y$.

As for uniqueness, we show in fact that the decomposition
$$
C_0, C_1, \cdots, C_{n-1}
$$
is unique up to a cyclic permutation. Suppose there is another decomposition $C'_0, C'_1,\cdots, C'_{n'-1}$ such that conditions (a)-(d) of $p_7$ hold. It follows from (a) that for each $i$, $C_i\cap C'_j\neq \emptyset$ for some $j$. Since $(C_i, \sigma^{nn'})$ and $(C_j', \sigma^{nn'})$ are both mixing and $C_i$, $C_j'$ are both relatively clopen, we must have $C_i=C'_j$. Thus, the uniqueness of the decomposition follows from (b).
\end{proof}

\begin{remark}
    We leave it to the reader to verify that for an irreducible sofic shift $Y$, our period $\p(Y)$ agrees with Adler's ``ergodic period'' given in~\cite{AM} (see page 6 and for an alternative version see pages 81-82).
\end{remark}

\begin{remark} \label{mixing_per=1} An irreducible sofic shift $Y$ is mixing iff $\p(Y) = 1$.  The ``if'' is obvious because in this case $Y$ would be a factor of a mixing SFT. For the ``only if,'' observe that if $\p(Y) > 1$, then considering the non-empty open sets $U_0,U_1$ from the proof of $p_6 = p_7$ above, for all multiples $n$ of $\p(Y)$, $\sigma^{n}(U_0) \cap U_1=  U_0 \cap U_1 = \emptyset$, a contradiction to mixing.
\end{remark}

\begin{remark}\label{13-07-25} As shown in the proof of $p_6 = p_7$, the collection of sets $C_i$ from $p_7$ is unique up to cyclic permutation. By Lemma 3.13 in \cite{T} the same holds for the sets $D_i$ in $p_6$. By Lemma \ref{05-07-25} the $D_i$'s are obtained from any cover as the images under the factor code of the sets in the canonical cyclic partition of the cover.	
\end{remark}

To conclude this section, we consider the relationships among various notions of period for irreducible SFTs and irreducible sofic shifts. For an irreducible SFT $X$, its period is the gcd of least periods of periodic points in $X$, which also equals the maximal $p$ such that $X$ is $p$-periodic. For a sofic shift $Y$, however, all these three things can be different:
\begin{enumerate}
    \item The period (which is the gcd of least periods of receptive periodic points) can be different from the gcd of least periods of all periodic points. By Remark \ref{mixing_per=1} any irreducible non-mixing sofic shift with a fixed point is an example of this. One simple example is given by the following labeled graph which also appears in \cite[page 3567, figure 1]{T}.

    \begin{tikzpicture}[scale=0.35]
				\draw [opacity=0] (0,0) grid (30,6);
				\node [circle, draw, thick] (a1) at (10,5) {};
				\node [circle, draw, thick] (a2) at (20,5) {};
				\draw [-stealth, black, thick] (a1) to [out=20, in=160] node[above] {$a$} (a2);
                \draw [-stealth, black, thick] (a2) to [out=200, in=340] node[below] {$a$} (a1);
				\draw [-stealth, black, thick] (a1) to [out=275, in=265] node[below] {$b$}  (a2);
			\end{tikzpicture}

    \item The gcd of least periods of periodic points can be different from the maximal $p$ for which $Y$ is $p$-periodic. See Example \ref{25-03-12-b}.
    \item The period can be different from the maximal $p$ for which $Y$ is $p$-periodic. See Example \ref{25-03-12-b}.
    \end{enumerate}

The following example illustrates both (2) and (3) above.

\begin{example} \label{25-03-12-b}
    Let $Y$ be the irreducible sofic shift described by the following labeled graph.

    
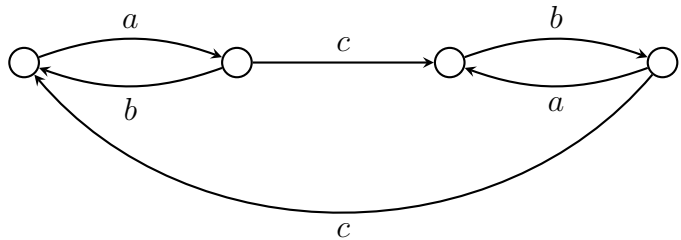
\begin{tikzpicture}[scale=0.35]
				\draw [opacity=0] (0,0) grid (30,8);
				\node [circle, draw, thick] (a1) at (0,6) {};
				\node [circle, draw, thick] (a2) at (8,6) {};
				\node [circle, draw, thick] (a3) at (16,6) {};
                \node [circle, draw, thick] (a4) at (24,6) {};
				\draw [-stealth, black, thick] (a1) to [out=20, in=160] node[above] {$a$} (a2);
                \draw [-stealth, black, thick] (a2) to [out=200, in=340] node[below] {$b$} (a1);
				\draw [-stealth, black, thick] (a2) to [out=0, in=180] node[above] {$c$}  (a3);
				\draw [-stealth, black, thick] (a3) to [out=20, in=160] node[above]{$b$}  (a4);
                \draw [-stealth, black, thick] (a4) to [out=200, in=340] node[below] {$a$} (a3);
                \draw [-stealth, black, thick] (a4) to [out=230, in=310] node[below]{$c$}  (a1);
			\end{tikzpicture}
		%\end{center}
		%\caption{A presentation of $Y$}
		%\label{25-03-12-a}
	%\end{figure}

     We claim that
  \begin{enumerate}
   \item[1.]    $Y$ is $p$-periodic only for $p=1$.
   \item[2.]  The $\gcd$ of the periods of the periodic points in $Y$ is 2.
   \item[3.]  $\p(Y)=2$.
\end{enumerate}

    We first prove item 1. Note that any subshift is $1$-periodic. Thus, it suffices to show $Y$ is not $k$-periodic for any $k\geq 2$. But since $Y$ contains a point of period 2, the only possibility would be that $Y$ is 2-periodic, with $Y = Y_0 \cup Y_1$, where $Y_0$ and $Y_1$ are closed and disjoint.
     Let $$x^{(n)}:= (c(ab)^{kn}.(ab)^{kn}ac(ba)^{4kn}b)^\infty,$$ WLOG infinitely many $x^{(n)} \in Y_0$. Since these points accumulate on $(ab)^\infty$, and  $Y_0$ is closed,
      $(ab)^\infty \in Y_0$   and so $(ba)^\infty \in Y_1$. Moreover, since $4kn +2$ is even,  $\sigma^{4kn+2}(x^{(n)}) \in Y_0$, and since these points  accumulate on $(ba)^\infty$, it follows that  $(ba)^\infty \in Y_0$, a contradiction to the disjointness of $Y_0$ and $Y_1$.

      Item 2 is clear from the graph.

      For item 3, first observe from the graph that $ca$ and $cb$ are synchronizing. Thus, with the  exception of $(ab)^\infty$, all periodic points contain synchronizing words and are thus receptive.  But $(ab)^\infty$ is receptive as well because $cb(ab)^nca$ is an allowed word for all
      $n$. Thus, all periodic points are receptive and so Item 3 follows from Item 2. \bigskip

\begin{remark} \label{Sturm}
		%\comm{klaus}{Any irreducible SFT with a fixed point is not $p$-peiodic for $p \geq 2$. So we need another excuse for focusing on Sturmian subshifts, and I suggest minimality.}
		%\comm{Tom}{How about this instead: If an irreducible SFT $X$ is not $p$-periodic, then $X$ necessarily has a periodic point with period not divisible by $p$. This property fails already for irreducible sofic shifts. In sharp contrast, there exists irreducible subshifts without any periodic points, such as Sturmian shifts, that are not $p$-periodic for any $p\geq 2$. To see this...  }
		%In contrast to irreducible SFTs, there exist irreducible subshifts, such as Sturmian shifts, that are not $p$-periodic for any $p\geq 2$.
		If an irreducible SFT $X$ is not $p$-periodic, then $X$ necessarily has a periodic point with period not divisible by $p$. This property fails already for irreducible sofic shifts, as we have seen from Example~\ref{25-03-12-b}. In fact, there exist irreducible subshifts without any periodic points, such as Sturmian shifts, that are not $p$-periodic for any $p\geq 2$. To see this, first recall that any Sturmian shift is uniquely ergodic and, with respect to the unique invariant measure $\mu$, it is isomorphic to an irrational rotation of the circle. It follows that with respect to $\mu$, the shift is totally ergodic, i.e., each power of the shift is ergodic.  But if the shift is $p$-periodic for some $p \ge 2$, then the shift space would decompose into $p$ subsets each of which is invariant under the $p$-th power of the shift and have measure $1/p$, contradicting total ergodicity. Thus, Sturmian shifts are not $p$-periodic for any $p\geq 2$. Note that this implies that the Sturmian shift can not embed into any non-mixing irreducible SFT by Proposition \ref{prop:irreducible_emb}. Moreover, this example shows that in general, no conditions on periodic points can ensure the shift is $p$-periodic. There also exist irreducible subshifts, such as Toeplitz shifts, that are $p$-periodic for infinitely many $p$ (see, for example, \cite[Lemma 2.3]{Wil84}).
	\end{remark}

\end{example}

\section{Almost invertible-factorizable Embeddings}	 \label{sec-AI-factorizable}

{Recall that Theorem \ref{16-03-23a} characterizes receptivity of a periodic point $y$ in an irreducible sofic shift $Y$ by the existence of a cover $\pi:X \to Y$ such that $\pi^{-1}(y)$ has a periodic point of the same least period as $y$. The proof is completely constructive.  However, the construction necessarily yields a
%n SFT 
covering SFT $X$
with $h(X) > h(Y)$.  This begs the question of whether one can achieve the same result with $h(X) = h(Y)$. We will show  that the answer is Yes.  In fact, we will show that there is a cover $\pi: X\to Y$ such that Theorem \ref{16-03-23a} holds and $\pi$ is almost invertible.

But first,  in the next example, we show that in general this cannot be done with a right or left-closing factor code  $\pi : X \to Y$; in particular it cannot always be done with the right or left Fischer cover. }

\begin{example}
	Let $Y$ be the sofic shift given by the vertex labeled graph $\mathcal G_1$ in Figure \ref{Fischer_cover_exmp}. 
    %Here the underlying SFT is given as a vertex shift, and it is well known that it can be easily converted to an edge shift (see, for example, \cite[Proposition 2.3.9]{LM}). 
    The symbols $u$ and $v$ are synchronizing for $Y$ and $ua^kv$ is an allowed word in $Y$ for all $k \in \mathbb N$. Thus, the fixed point $a^\infty$ is receptive. One checks that $\mathcal G_1$ is irreducible, right-resolving and follower-separated. Thus, $\mathcal G_1$ is the right Fischer cover of $Y$. Observe that in $\mathcal G_1$, all bi-infinite paths representing $a^\infty$ are cycles of minimal length at least $2$. Now, by \cite[Proposition 5.1.11]{LM} and \cite[Corollary 3.3.20]{LM}, any right-closing cover of $Y$ must factor through $\mathcal G_1$, so for any right-closing cover of $Y$, the preimage of $a^\infty$ has least period strictly larger than $1$.
	
	% \comm{klaus}{I don't think Prop. 5.1.11 in [LM] is the right reference here. I found the required statement in Brians paper with Boyle and Kitchen from '85, 'Minimal covers of sofic systems'. I suggest to change the reference accordingly, unless there is a statement or exercise in [LM] which can work. 
		%
		%    This example and also the proof of the AI-lemma makes it necessary to explain the notion of labelled vertex shift somewhere.}
	
	%Now, from $G_1$, there is a general procedure to find the left Fischer cover of $Y$ (see, for example, the proof of \cite[Lemma 3.3.8]{LM}) %and we leave it to the reader that the resulting left Fischer cover can be represented by the labeled graph $G_2$ in Figure \ref{Fischer_cover_exmp}. 
    Let $\mathcal G_2$ be the labeled graph given in Figure \ref{Fischer_cover_exmp}. The reader can verify that $\mathcal G_2$ is irreducible, left-resolving and predecessor-separated. Using the algorithm from the proof of \cite[Theorem 3.4.13]{LM}, one verifies that the labeled graph $\mathcal G_1$ and $\mathcal G_2$ represents the same sofic shift. Alternatively, the reader can directly verify that the set of bi-infinite label sequences of both $\mathcal G_1$ and $\mathcal G_2$ is exactly {the closure of} the set of concatenations of blocks of the form $ua^nv$, $uwa^{2n}v$ and $ua^{2n+1}$ where $n\in \mathbb{N}$. {This can be shown using the fact that the only bi-infinite label sequences in both $\mathcal{G}_1$ and $\mathcal{G}_2$ that are not concatenations of  blocks of the form $ua^nv$, $uwa^{2n}v$ and $ua^{2n+1}$  
    are those that are left or right asymptotic to $a^\infty$.}
    Thus, $\mathcal G_2$ is indeed the left Fischer cover of $Y$.

    Since $a^\infty$ is represented by cycles of length at least $2$ in $\mathcal G_2$, by a similar argument as in the previous paragraph, we conclude that for any left-closing cover of $Y$, the preimages of $a^\infty$ have least period strictly larger than $1$.
	
		
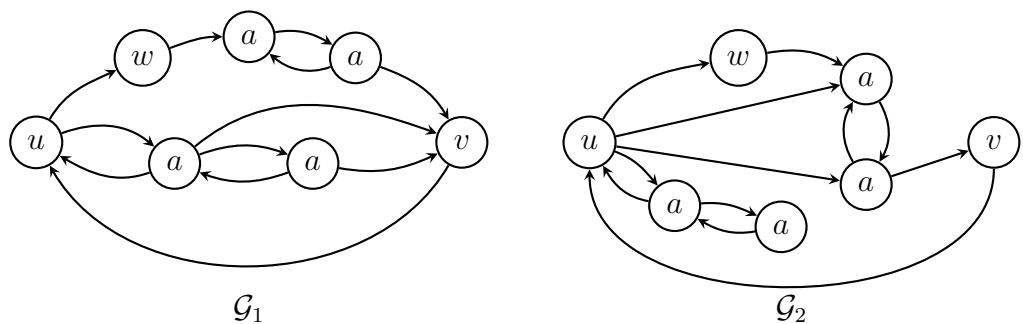
\begin{center}
			\begin{tikzpicture}[scale=0.28]
				\draw [opacity=0] (0,0) grid (50,17);
				\node [circle, draw, thick] (a1) at (0,10) {$u$};
				\node [circle, draw, thick] (a2) at (5,14) {$w$};
				\node [circle, draw, thick] (a3) at (10,15) {$a$};
				\node [circle, draw, thick] (a4) at (15,14) {$a$};
				\node [circle, draw, thick] (a5) at (20,10) {$v$};
				\node [circle, draw, thick] (a6) at (6.5,9) {$a$};
				\node [circle, draw, thick] (a7) at (13,9) {$a$};
				\draw [-stealth, black, thick] (a1) to [out=60, in=205]  (a2);
				\draw [-stealth, black, thick] (a2) to [out=20, in=180]  (a3);
				\draw [-stealth, black, thick] (a3) to [out=10, in=150]  (a4);
				\draw [-stealth, black, thick] (a4) to [out=200, in=320]  (a3);
				\draw [-stealth, black, thick] (a4) to [out=340, in=120]  (a5);
				\draw [-stealth, black, thick] (a1) to [out=20, in=140]  (a6);
				\draw [-stealth, black, thick] (a6) to [out=200, in=330]  (a1);
				\draw [-stealth, black, thick] (a6) to [out=20, in=160]  (a7);
				\draw [-stealth, black, thick] (a7) to [out=200, in=340]  (a6);
				\draw [-stealth, black, thick] (a6) to [out=40, in=160]  (a5);
				\draw [-stealth, black, thick] (a7) to [out=350, in=205]  (a5);
				\draw [-stealth, black, thick] (a5) to [out=240, in=300]  (a1);
				
				\node [circle, draw, thick] (b1) at (26,10) {$u$};
				\node [circle, draw, thick] (b2) at (33,14) {$w$};
				\node [circle, draw, thick] (b3) at (39,13) {$a$};
				\node [circle, draw, thick] (b4) at (39,8) {$a$};
				\node [circle, draw, thick] (b5) at (45,10) {$v$};
				\node [circle, draw, thick] (b6) at (30,7) {$a$};
				\node [circle, draw, thick] (b7) at (35,6) {$a$};
				\draw [-stealth, black, thick] (b1) to [out=60, in=190]  (b2);
				\draw [-stealth, black, thick] (b5) to [out=270, in=270]  (b1);
				\draw [-stealth, black, thick] (b2) to [out=10, in=150]  (b3);
				\draw [-stealth, black, thick] (b3) to [out=300, in=60]  (b4);
				\draw [-stealth, black, thick] (b4) to [out=120, in=240]  (b3);
				\draw [-stealth, black, thick] (b4) to  (b5);
				\draw [-stealth, black, thick] (b1) to  (b3);
				\draw [-stealth, black, thick] (b1) to  (b4);
				\draw [-stealth, black, thick] (b1) to [out=340, in=130]  (b6);
				\draw [-stealth, black, thick] (b6) to [out=170, in=300]  (b1);
				\draw [-stealth, black, thick] (b6) to [out=5, in=155]  (b7);
				\draw [-stealth, black, thick] (b7) to [out=190, in=330]  (b6);
				\node at (10, 1) [coordinate, draw, fill=black, label=above: $\mathcal G_1$] {};
				\node at (35.5, 1) [coordinate, draw, fill=black, label=above: $\mathcal G_2$] {};
				%\node at (9, 1.7) [coordinate, draw, fill=black, label=above: $\gamma^+$] {};
				%\node at (9, 0.3) [coordinate, draw, fill=black, label=below: $\gamma^-$] {};
			\end{tikzpicture}
		\end{center}
%		\caption{Left and right Fischer cover of $Y$}
		\label{Fischer_cover_exmp}
\end{example}

An important fact that we will use later on is:

\begin{prop} \label{synchro_infinitelyoften}
    Let $Y$ be an irreducible sofic shift and $\pi: X\to Y$ be the Fischer cover of $Y$. For any point $y\in Y$, if synchronizing words appear infinitely often to the left in $y$, then $y$ has a unique $\pi$-preimage.
  %  In particular, every synchronizing periodic point in $Y$ must have a unique $\pi$-preimage.
\end{prop}
\begin{proof}
First recall from Lemma \ref{19-03-23d} that synchronizing words are the same as magic words under the Fischer cover. Then since the Fischer cover is right-resolving, it is not hard to see $y$ has a unique preimage.
\end{proof}

A periodic point $y$ of least period $n$ in an irreducible sofic shift $Y$ is called {\em synchronizing} if it contains a synchronizing word, equivalently if there exists $k\in \mathbb{N}$ such that $(y_{[0,n-1]})^k$ is synchronizing.

We will need to know that the asymptotic growth rate of the number of synchronizing periodic points with least period $np$ is $h(Y)$, where $p$ is the period of $Y$. This is a part of Proposition~\ref{growth} below, {and also follows from \cite[Lemma 3.1 and Theorem 3.2]{T}}

{Let $S_n(Y)$ denote the set of synchronizing periodic points of least period $n$ in $Y$ and 
$
s_n(Y):= \# S_n(Y).
$
Recall that $q_n(Y)$ denotes the number of periodic points with least period $n$ in $Y$,   
$\rec_n(Y)$ denote the number of receptive periodic points of least period $n$ in $Y$, and
for a cover $\pi:X \to Y$, $r_n(\pi)$ denotes the number of periodic points with least period $n$ in $Y$ which have a preimage with the same least period.  }
%\comm{Chengyu}{Check notations for receptive points.}
\begin{prop} \label{growth}
Let $Y$ be an irreducible sofic shift with period $p$. 
\begin{enumerate}
\item[(a)] 
$\rec_n(Y) \le q_n(Y)$ 
\item[(b)] 
For any cover $\pi: X \to Y$, 
$r_n(\pi) \le \rec_n(Y)$ 
\item[(c)] 
For the right or left Fischer cover $\pi$,  
$s_n(Y) \le r_n(\pi)$
\item[(d)] 
For $a_n$ equal to any of $q_n(Y), \rec_n(Y), s_n(Y)$ or $r_n(\pi)$ for the right or left Fischer cover of 
$Y$, 
$$
\limsup_{n \to \infty} \frac{1}{n}\log a_n = 
\lim_{n \to \infty}\frac{1}{pn} \log a_{pn} = h(Y).
$$
\end{enumerate}
\end{prop}

\begin{proof}
Item (a) is obvious, and Item (b) follows from Theorem~\ref{16-03-23a}.  Item (c) follows from 
%the fact that for the right and left Fischer covers, magic words coincide with synchronizing words~\cite[Lemma 4.3]{MMTW}, along with 
Proposition~\ref{synchro_infinitelyoften}. 

Now, we turn to Item (d). By~\cite[Corollary 4.3.8]{LM},
\begin{equation}
\label{limsup}
\limsup_{n \to \infty} \frac{1}{n}\log q_n(Y)  = h(Y).
\end{equation}
By Lemma \ref{growthrate_rec}, 
$$
\lim_{n \to \infty} \frac{1}{np}\log r_{np}(\pi)  = h(Y).
$$
It follows that for $a_n$ 
 equal to any of $q_n(Y), r_n(Y)$ or $r_n(\pi)$ for the right or left Fischer cover,
 $$
\limsup_{n \to \infty} \frac{1}{n}\log a_n  = 
\lim_{n \to \infty} \frac{1}{np}\log a_{np}  = h(Y).
$$
Finally, we consider $s_n(Y)$.  By~(\ref{limsup}) above and items (a),(b) and (c), we have 
\begin{align} \label{limsup_sn}
\limsup_{n \to \infty} \frac{1}{n}\log s_n(Y)  \le  h(Y).
\end{align}
Let $u$ be a synchronizing word of $Y$. Let $Y'$ be the subshift obtained by further forbidding $u$ from $Y$. Then $h(Y')<h(Y)$ by \cite[Corollary 4.4.9]{LM}. Thus, by \cite[Proposition 4.1.15]{LM}
	$$
	\limsup_{n\to \infty } \frac{1}{np} \log q_{np}(Y') \leq h(Y')<h(Y), 
	$$
    %where $Q_{np}(Y')$ is the set of periodic points in $Y'$ with least period $np$.
   It follows that given any $\epsilon>0$, for sufficiently large $n$,  $q_{np}(Y')\leq \epsilon q_{np}(Y)$.
	Then, since any point in $Q_{np}(Y)\setminus Q_{np}(Y')$ must contain the synchronizing word $u$ (and therefore it is a synchronizing periodic point), we have
	\begin{align} \label{liminf_sn}
	\liminf_{n\to\infty} \frac{1}{np} \log s_{np}(Y)&\geq \liminf_{n\to \infty}\frac{1}{np} \log (q_{np}(Y)-q_{np}(Y')) \\
    &\geq \liminf_{n\to \infty}\frac{1}{np} \log (q_{np}(Y)(1-\epsilon)) = h(Y). \notag 
	\end{align}
    Item (d) then follows by combining (\ref{limsup_sn}) and (\ref{liminf_sn}).
	%where the last inequality follows from %$\lim\limits_{n\to \infty}\frac{1}{np}\log %q_{np}(Y)= h(Y)$.
\end{proof}

The following result will be the key step in showing that the cover achieved in 
Theorem~\ref{16-03-23a} can be achieved by 
an almost invertible cover $\pi:X \to Y$, in particular a cover with $h(X) = h(Y)$.  See 
Corollary~\ref{equal_entropy} below. 

\begin{prop} \label{AIcover_sofic}
Let $Y$ be an irreducible sofic shift with a receptive periodic point $\xi$ of least period $n$ and $h(Y)>0$. Then there exists an irreducible sofic $\hat{Y}$ and an almost invertible code $\hat{\pi}: \hat{Y}\to Y$ such that 
\begin{enumerate}
	\item[1.] there is a synchronizing periodic point $\hat{\xi}\in \hat{Y}$ such that $\hat{\pi}(\hat \xi)= \xi$ and $\per(\hat{\xi})=\per(\xi)$;
	
	\item[2.] for any periodic point $\eta\in Y$ with $\eta \neq \xi $, the following hold: 
	\begin{itemize}
		\item[(a)] $\eta$ has a unique $\hat{\pi}$-preimage (call it $\tilde{\eta}$); and 
		\item[(b)] if $\eta$ is synchronizing in $Y$, then $\tilde{\eta}$ is synchronizing in $\hat{Y}$.
	\end{itemize}	
\end{enumerate}
\end{prop}	
\noindent{(Remark 1: for item 2(a), we will indeed prove something stronger: any point $y\in Y$ that is not left/right asymptotic to $\xi$ has a unique $\hat{\pi}$-preimage.)}

%	Remark 2: the folowing may be useful for proving 2(b): Let $X$ be sofic and $u$ be synchronizing, then there is a list of forbidden words $\mathcal{F}$, all of which start with $u$ or end with $u$, such that $X=X_\mathcal{F}$.)}

\begin{proof}
%Description of $\hat{Y}$ and $\hat{p}^\infty$: 
%\begin{itemize}
%	\item enlarge alphabet by $\hat p_0, \cdots \hat p_{n-1}$; and
%	\item the set of forbidden words of $Y$ are contained in the set of forbidden words of $\hat{Y}$; 
%	\item forbid $u(p_1\cdots p_{n-1})^k v$ for all $k$; and
%	\item for each $0\leq i\leq p-1$, if $s\hat p_i$ is allowed, then $s=\hat p_{i-1 \bmod p}$ or $i=0$ and $s=u$;
%	\item for each $0\leq i\leq p-1$, if $\hat p_i t$ is allowed, then $t=\hat p_{i+1 \bmod p}$ or $i=p-1$ and $t=v$.
%\end{itemize}	

Denote $w:=\xi_{[0,n-1]}$, so that $\xi=w^\infty$. 
Let $u$ and $v$ be synchronizing words such that $uw^k v\in \mathcal{B}(Y)$ for all $k\geq 1$. By extending $u$ (to the left) and $v$ (to the right) if needed, we may assume $\vert u \vert = \vert v \vert>n$, and the prefix of $u$ of length $n$ and the suffix of $v$ of length $n$ is not in $\{\sigma^j(\xi)_{[0,n-1]}: 0\leq j\leq n-1\}$.

Let $Y'$ denote the subshift of $Y$ obtained by forbidding the blocks $uw^kv$ for all $k\geq 1$. We claim  that $Y'$ is sofic.
To see this, first observe that we can write $Y' = Y \cap Z$, where $Z$ is the shift space obtained by forbidding, from the full shift,
words of the form  $uw^kv$.  Now  $Z$ is sofic  because the set   $\{uw^kv: k\in \mathbb{N}\setminus \{0\}\}$ is a regular language, i.e., the set of all outputs of a finite directed labelled graph which all start in the same initial state and end in the same terminal state.
Then $Y'$ is sofic because the intersection of sofic shifts is sofic {\cite[Proposition 3.4.10]{LM}}.

Let $\mathcal{A}$ be the alphabet of $Y$ and let $\{\hat{w}_0, \cdots, \hat{w}_{n-1}\}$ be a set of distinct elements that is disjoint from $\mathcal{A}$. {Let $\hat{w}:=\hat{w}_0\hat{w}_1\cdots \hat{w}_{n-1}$ and $\hat{\xi}:=\hat{w}^\infty.$}
%Denote $\hat{\mathcal A}: = \mathcal{A} \cup \{\hat{p_0}, \cdots, \hat{p}_{n-1}\}$.

Let $Y''$ be the set of bi-infinite sequences of the form
\begin{equation} \label{points_in_Y''}
	\ldots \alpha_i \hat{w}^{k_i} \alpha_{i+1} \hat{w}^{k_{i+1}}\ldots
\end{equation}	
where each $k_i\geq 1$ and each $\alpha_i$ is a block allowed in $Y'$ and begins with $v$ and ends with $u$.
%\comm{Klaus}{I think you want $k_i \geq 1$? Otherwise you will introduce the word $uv$ which may not be in $Y$.}  
In other words, to obtain $Y''$ from $Y'$,  one adds stretches of full periods of $\hat{w}$  that can escape into stretches of $Y'$ only through $u$ and $v$.

Let $\hat{Y}  = \overline{Y''}$ {be the closure of $Y''$ in $\mathcal B^\mathbb Z$, where $\mathcal B =\mathcal A \cup \{\hat{w}_0, \cdots, \hat{w}_{n-1}\}$}. So $\hat{Y}$ augments $Y''$ by adding in points that begin and/or end with semi-infinite sequences
$\hat{w}^\infty$ or semi-infinite sequences allowed in $Y'$. We claim $\hat{Y}$ is sofic. To see this,
{let $\mathcal{G}=(G, \mathcal{L})$ be a vertex labeled graph representing $Y'$. Let $\mathcal{H}$ be the vertex labeled graph defined by the following:
\begin{itemize}
    \item The vertex set of $\mathcal{H}$ is the set of vertices of $G$ with the same labels and $n$ new vertices labeled by $\{\hat{w}_0, \cdots, \hat{w}_{n-1}\}$;
    \item For any pair $(I,J)$ of vertices from $\mathcal{H}$, $IJ$ is an edge in $\mathcal H$ if and only if either $IJ$ is an edge in $G$, or at least one of $I$ and $J$ is one of the $n$ new vertices defined above.
\end{itemize}
Let $Y^*$ be the sofic shift represented by $\mathcal{H}$.
%We show that $\hat{Y} = Y^*\cap W$ where 
%$W$ is an SFT. Here,
Let $W$ be the SFT over the alphabet $\mathcal{B}$ given by the following rules:
%that require the symbols $\hat{w}_i$ to appear in the correct cyclic order and that stretches of full periods of $\hat{w}$
%can only enter into stretches of $Y'$ from $v$, and leave stretches of $Y'$ from $u$. More precisely, $W$ is the SFT over $\mathcal{A}\cup \{\hat w_0, \cdots \hat w_{n-1}\}$ define by the following rule: 
for any point $x\in W$, if $x_i=\hat{w}_j$, then
\begin{enumerate}
    \item if $1 \leq j\leq n-2$, we have $x_{i+1}=\hat w_{j+1}$;
    \item if $j=n-1$, then either $x_{i+1}=\hat w_0$ or $x_{i+1}\cdots x_{i+\vert v \vert}=v$;
    \item if $j=0$, then either $x_{i-1}=\hat w_{n-1}$ or $x_{i-\vert u \vert} \cdots x_{i-1}=u$.
\end{enumerate}
It is not hard to see that $\hat{Y}=Y^*\cap W$. Therefore, $\hat{Y}$ is sofic.}

We then claim that $u$ and $v$ are synchronizing in $\hat{Y}$.
{We first prove $u$ is synchronizing.}
%WLOG we consider only $u$.
Since $Y'$ is a subshift of $Y$, $u$ is synchronizing in $Y'$.  It is enough to show that $u$ is synchronizing in $Y''$ in the sense that if $xu$ and $uy$ are words that occur in  elements of $Y''$, then so does $xuy$. To this end, by extending $x$ to the left if necessary, we may write $x = ab$, where $bu$ is a word in $Y'$ that begins with $v$ and $a$ is an alternating concatenation of words in $Y'$, each beginning with $v$ and ending with $u$, and powers of $\hat{w}$, ending with the latter.
Similarly, we may write $y = cd$, where $uc$ is a word in $Y'$ that ends with $u$ ($c$ is possibly empty) and $d$ is an alternating concatenation of words in $Y'$, each
beginning with $v$ and ending with $u$, and powers of $\hat{w}$, beginning with the latter.
Since $u$ is synchronizing in $Y'$, $buc$ is allowed in $Y'$, begins with $v$ and ends with $u$.  But then $xuy= abucd$ is of the form of a finite  word in some element of $Y''$ and  thus is allowed in $Y''$ and hence in $\hat{Y}$. {Thus, $u$ is a synchronizing word in $\hat{Y}$. The proof that $v$ is synchronizing in $\hat{Y}$ is similar.}

Let $\hat{\pi}: \hat{Y}\to Y$ be a sliding block code induced by the $1$-block map $\Pi: \hat{\mathcal{A}} \to \mathcal{A}$
defined by 
$$
\Pi(a) =
\begin{cases}
a \qquad \mbox{if $a\notin \{\hat{w}_0,\cdots, \hat{w}_{n-1}\}$} \\
w_i \qquad \mbox{if $a=\hat{w}_i$, where $0\leq i\leq n-1$}.
\end{cases}
$$	

%Note that $\hat{Y}$ is a sofic shift since it is a intersection of three sofic shifts. Moreover, any word in $\mathcal{B}(Y)$ is allowed in $\mathcal{B}(\hat{Y})$ if and only if it does not contain any word from $\mathcal{F}_2$. 

We first show $\hat{\pi}$ maps $\hat{Y}$ into $Y$. Since $Y$ is closed and $\hat{\pi}$ is continuous, it suffices to show $\hat{\pi}$ maps $Y''$ into $Y$. 
Take any point $y \in Y''$. Note that $y$ is of the form (\ref{points_in_Y''}) where each $\alpha_i$ is in $\mathcal{B}(Y')$ and begins with $v$ and ends with $u$. By definition, $\hat{\pi}(y)$ is
\begin{equation} \label{points_in_Y_0}
\ldots \alpha_i w^{k_i} \alpha_{i+1} w^{k_{i+1}} \ldots.
\end{equation}
Since for each $i$, $\alpha_i$ and  $uw^{k_i}v$ are allowed in $Y$, we have $\hat{\pi}(y)\in Y$ because $u$ and $v$ are synchronizing in $Y$.

%Note that $y$ is an alternating concatenation of words in $Y'$, each starting with $v$ and ending with $u$, and words of the form $(\hat{p}_{[0,n-1]})^k$ for some $k$. Now $\hat{\pi}(y)$ is obtained from $y$ by replacing each $(\hat{p}_{[0,n-1]})^k$ with $(p_{[0,n-1]})^k$, and it must be a legal point in $Y$ because $u (p_{[0,n-1]})^k v\in \mathcal{B}(Y)$ and $u, v$ are synchronizing in $Y$.

%Thus, it must be a legal point in $Y$ because $u$ and $v$ are synchronizing in $Y$.

We then show $\hat\pi$ maps $\hat{Y}$ onto $Y$. Let $Y_0\subset Y$ be the set of bi-infinite sequences of the form (\ref{points_in_Y_0}), where each $\alpha_i$ is allowed in $Y'$ starting with $v$ and ending with $u$.  Since $\hat{Y}$ is compact, $Y_0$ is dense in $Y$ and $\hat{\pi}$ is continuous, it suffices to show $\hat{\pi}$ maps $Y''$ onto $Y_0$. But this follows immediately from the definition of $Y''$ and $\hat{\pi}$.

We claim that $\hat{Y}$ is irreducible. Since $Y''$ is dense in $\hat{Y}$ \footnote{{For a {shift invariant} subset $X'$ of a shift space $X$, we adapt the concepts of language, irreducible and mixing to $X'$ (defined in the same ways) and observe that if $X'$ is dense in $X$, then ${\mathcal B}(X') = {\mathcal B}(X)$, and $X'$ is irreducible (resp. mixing) iff $X$ is irreducible (resp. mixing).}}, it suffices to show $Y''$ is irreducible. %(say something about the irreducibility of a subset of shift spaces, may be in the notation section). 
To see this, let $x$ and $y$ be two arbitrary allowed words in $Y''$. By the definition of $Y''$, there exist words $s$ and $t$ such that $xsu$ and $uty$ are allowed in $Y''$. Since $u$ is synchronizing for $Y''$, we have $xsuty$ is allowed in $Y''$. This shows $Y''$ (and therefore $\hat{Y}$) is irreducible.  
%Let $a$ and $b$ be two symbols in $\hat{\mathcal{A}}$. We first claim that there exist $r, s\in \mathcal{B}(\hat{Y})$ such that $aru\in \mathcal{B}(\hat{Y})$ and $vsb\in \mathcal{B}(\hat{Y})$. To see this, consider two cases: if $a\in \mathcal{A}$, then, by irreducibility of $Y$, there exist words in $\mathcal{B}(Y)$  such that $vtaru$ is allowed in $Y$ and  $u(p_{[0,n-1]})^k v$ does not appear in $vtaru$ for all $k$. Thus, $vtaru$ appears in some point in $Y''$. In particular, $aru$ is in allowed in $Y''$; if $a=\hat{p}_i$ for some $0\leq i\leq n-1$, first note that $\hat{p}_{i}\cdots \hat{p}_{n-1} v$ is allowed in $Y''$. By irreducibility of $Y$,  there is a word $r'\in \mathcal{B}(Y)$ such that $vr'u\in \mathcal{B}(\hat{Y'})$ and $u(p_{[0,n-1]})^kv$ does not appear in $vr'u$. Since $v$ is synchronizing, we have $a\hat{p}_i\hat{p}_{i+1}\cdots \hat{p}_{n-1} vr'u\in \mathcal{B}(\hat{Y})$. Let $r:= \hat{p}_{i+1}\cdots \hat{p}_{n-1} vr'$. Then $aru\in \mathcal{B}(Y)$. The proof of the existence of $s$ is simialr.
%Now since $aru$ and $vsb$ are both in $\mathcal{B}(\hat{Y})$, and $u \hat{p}_0\cdots \hat{p}_{n-1} v\in \mathcal{B}(\hat{Y})$, we have $aru\hat{p}_0\cdots \hat{p}_{n-1}vsb\in \hat{Y}$ since $u$ and $v$ are synchronizing in $\hat{Y}$. This shows $\hat{Y}$ is irreducible.

We then prove part 1 of Proposition \ref{AIcover_sofic}. Recall that $\xi=w^\infty$ and $\hat{\xi}=\hat{w}^\infty$. 
It is clear from the construction of $\hat{Y}$ that $\hat{w}^\infty\in \hat{Y}$ and $\per(\hat w^\infty)=\per(w^\infty)$. Moreover, we have $\hat{\pi}(\hat w^\infty)=w^\infty$. To show $\hat w^\infty$ is a synchronizing periodic point in $\hat{Y}$, it suffices to show $\hat{w}$ is a synchronizing word of $Y''$, in the sense that if $x \hat w$ and $\hat wy$ are allowed words in $Y''$, then so is $x \hat wy$. By extending $x$ to the left if necessary, we may write $x:=au\hat w^{\ell_1}$ where $\ell_1\geq 0$ and $a$ is an allowed word in $Y''$. Similarly, we may write $y:=\hat w^{\ell_2}vb$ where $\ell_2\geq 0$ and $b$ is an allowed word in $Y''$. Note that $u\hat{w}^{\ell_1+\ell_2+1} v$ is allowed in $Y''$. Since $u$ and $v$ are both synchronizing in $Y''$, $au\hat{w}^{\ell_1+\ell_2+1} vb$ must be allowed in $Y''$, i.e., $x\hat{w}y$ is allowed in $Y''$. Thus, $\hat{w}$ is a synchronizing word and therefore $\hat{w}^\infty$ is a synchronizing periodic point in $Y''$.
%If $x=\hat{p}^\infty$ and $y=\hat{p}^\infty$ ($x$ and $y$ are left and right infinite rays, respectively), then $x \hat p _{[0,n-1]} y=\hat p^\infty$ is allowed in $\hat{Y}$; otherwise, suppose that $x\neq \hat p^\infty$. Then, there exists $k\in \mathbb{N}$ such that $u\hat{p}_{[0,m-1]}^k$ is a suffix of $x$. Write $x=x'u\hat{p}_{[0,m-1]}^k$. Note that $u\hat{p}_{[0,m-1]}^k y$ is allowed in $\hat{Y}$ by the construction of $\hat{Y}$. Now, since $u$ is synchronizing for $\hat{Y}$, we have $x'u\hat{p}_{[0,m-1]}^k y$ is in $\hat{Y}$, proving that $\hat{p}_{[0,m-1]}$ is synchronizing for $\hat{Y}$.

To show $\hat{\pi}$ is almost invertible and item 2(a) of the proposition, we 
first claim that any $y\in Y$ that is not left/right asymptotic to $w^\infty$ must have a unique preimage. To this end, take such a $y$. Recalling that the prefix of $u$ of length $n$ and the suffix of $v$ of length $n$ is not in $\{\sigma^j(w^\infty)_{[0,n-1]}: 0\leq j\leq n-1\}$, we know that $y$ can be uniquely written as 
$$
y= \ldots \alpha_i {w}^{k_i} \alpha_{i+1} {w}^{k_{i+1}} \ldots
$$
where for each $i$, $1\leq k_i<\infty$ and  $\alpha_i$ is an allowed word in $Y'$ that begins with $v$ and ends with $u$, {except that $y$ may begin and/or end with semi-infinite sequences allowed in $Y'$}. In particular, each $\alpha_i$ does not contain $uw^kv$ for any $k\geq 1$. Thus, the only $\hat{\Pi}$-preimage of each $\alpha_i$ is $\alpha_i$ itself, and the only $\hat{\Pi}$-preimage of each $w^k$ is $\hat w^k$ because it is preceded by a $u$ and followed by a $v$. This proves that  $y$ has a unique $\hat{\pi}$-preimage. %\comm{Klaus}{For the proof of the following Cor. it seems necessary to know that the pre-image of a doubly transitive point is itself doubly transitive. Is this trivial or is there another way to show that composition in the proof of the Corollary is almost invertible? Lemma 9.1.13 in [LM] would do the trick if we have a sofic version of Prop. 9.2.2 in [LM]. Do we? - Anyway, I think something is missing here or in the proof of the Corollary.}
%\comm{Brian}{Klaus, good point! Yes, there is a sofic version of Prop. 9.2.2 in [LM]. The question you raise is whether the  almost invertible factor code $\hat{\phi}$ is finite-to-one. One can see this directly for the given $\hat{\phi}$ in the proof, which I believe is at most 4-to-1. But  fact, every almost invertible code $\hat{\phi}$ on an irreducible sofic $Y$ is finite-to-one. Proof: Let $\pi$ be the Fischer cover of the domain of $\hat{\phi}$.
%Since $\pi$ is finite to one,  every doubly transitive point in the range of $\hat{\phi}$ has finitely many preimages via $\hat\phi \circ \pi$, whose domain is an irreducible SFT.   Thus, 
%$\hat{\phi} \circ \pi$ is finite-to-one  because if not it would have a diamond from which you can construct a doubly transitive point with infinitely many preiamges. Thus $\hat{\phi}$ is also finite-to-one.}

%then partition $y$ into segments such that either it is of the form $u(p_1\cdots p_{n-1})^kv$ for some $k$ or otherwise. For segments of the type $u(p_1\cdots p_{n-1})^kv$, then its $\hat \Pi$- preimage is $u(\hat p_1\cdots \hat p_{n-1})^kv$; if the segment is not of this type, then it also has a unique $\hat{\Pi}$-preimage by construction. This means we can uniquely decode $y$ and therefore the claim is proved. 

Item 2(a) of the proposition directly follows from the claim above.
Moreover, since any doubly transitive point in $Y$ is not left/right asymptotic to $w^\infty$, it must {have} a unique preimage and therefore $\hat{\pi}: \hat{Y}\to Y$ is almost invertible. 

It remains to show 2(b). {Let $\eta$ and $\tilde\eta$ be as in the statement of item 2. Note that $\eta$ and $\tilde{\eta}$ must have the same least period. We use $m$ to denote this least period. We also denote $r:=\eta_{[0,m-1]}$ and $\tilde{r}:=\tilde{\eta}_{[0,m-1]}$, so that $\eta=r^\infty$ and $\tilde{\eta}=\tilde{r}^\infty$.
Since $r^\infty$ is a synchronizing periodic point,
%Denote $r^\infty= r_{[0,m-1]}^\infty$ and $\tilde{r}^\infty= \tilde{r}_{[0,m-1]}^\infty$ and 
there is a $k\in \mathbb{N}$ such that $r^k$ is a synchronizing word in $Y$.} %We may assume $mk>2n$. 
%To prove item 2(b), it suffices to show  $(\tilde{r}_{[0,m-1]})^k$ is sinchronizing in $Y''$.
Now consider two cases:

If for some $0\leq i\leq n-1$, $\hat{w}_i$ appears in $\tilde{r}^k$, then $\hat{w}$ must appear in $\tilde{r}^k$ and therefore $\tilde{r}^k$ must be synchronizing in $\hat{Y}$ because it is an extension of the synchronizing word $\hat{w}$.

If $\hat{w}_i$ does not appear in $\tilde{r}^k$ for any $0\leq i\leq n-1$, then $\tilde{r}^k=r^k$ and consequently $\tilde{r}^k$ is synchronizing in $Y$ because ${r}^k$ is. We now show $\tilde{r}^k$ is synchronizing in $\hat{Y}$, {which is equivalent to showing that ${r}^k$ is synchronizing in $Y''$. We may assume that both $u$ and $v$ do not appear in $r^k$, because otherwise $r^k$ will be automatically synchronizing in $Y''$.} Let $x{r}^k$ and ${r}^ky$ be words that occur in elements of $Y''$. By extending $x$ to the left if necessary, we may write $x{r}^k=aub$, where $ub$ is an allowed word in $Y'$ containing ${r}^k$ as a suffix, and $a$ is an allowed word in $Y''$. Similarly, we may write ${r}^ky=cvd$ where $cv$ is an allowed word in $Y'$ containing ${r}^k$as a prefix, and $d$ is an allowed word in $Y''$. Now, using the synchronization of ${r}^k$ in $Y'$, we can ``glue" words $ub$ and $cv$ to form an allowed word in $Y''$. Call this resulting word {{$t$, which has $u$ as a prefix and $v$ as a suffix. Finally, since $au, t, vd$ are allowed in $Y''$ and $u, v$ are synchronizing in $Y''$, we conclude that $atd$}} is allowed in $Y''$, but $x{r}^ky=atd$. Thus, ${r}^k$ is synchronizing in $Y''$ {and the proof is complete}. 
\end{proof}					 	

%\begin{prop} \label{AI_extension_corollary}
%Let $Y$ be an irreducible sofic shift. Then there exist an irrecucible SFT $X$ and an almost invertible code $\pi: X\to Y$ such that $\rec_n(Y)\leq r_n(\pi)$ for all $n$.
%\end{prop}
%\begin{proof}
%Since $h$
%$N$ Let $w^{(1)}, \cdots, w^{(m)}$ be the set of receptive periodic points in $Y$ of least period $n$. Applying Proposition \ref{AIcover_sofic} recursively with $p^\infty$ replaced by $w^{(i)}$ at each step, we obtain a sofic shift $\tilde{Y}$, an almost invertible map $\hat{\pi}: \hat{Y}\to Y$ and periodic points $\{\hat{w}^{(1)}, \hat{w}^{(2)]},\cdots, \hat{w}^{(m)}\}$ in $\hat{Y}$ such that for all $1\leq i \leq m$, $\hat{\pi}(\hat{w}^{i})=w^{(i)}$ and each $\hat{w}^{(i)}$ is a synchronizing periodic point.

%Now let $\phi: X\to \hat{Y}$ be the Fischer cover of $\hat{Y}$. Note that any word in $Y$ is synchrozing if and only if it is magic for $\phi$ (see, for example, \cite[Lemma 4.3]{}). Thus, each $\hat{w}^{(i)}$ is a periodic point containing a magic word and therefore has a unique $\phi$-preimage. Note in particular that the this preimage has the same least period as $\hat{w}^{(i)}$. 

%Finally, let $\pi:= \phi\circ \hat{\pi}$. Then it is an almost invertible code from $X$ onto $Y$, and each receptive periodic point in $Y$ has a $\pi$-preimage with the same least period. Thus, $\rec_n(Y)\leq r_n(\pi)$.
%\end{proof}	

As a Corollary, we get the following characterizations of receptivity in terms of SFT covers. 

\begin{cor}\label{equal_entropy} Let $Y$ be an irreducible sofic shift and $y \in Y$ a periodic point.  The following are equivalent.
\begin{enumerate}
\item 
$y$ is receptive. 
\item 
There is a cover $\pi:X\to Y$
%a factor code $\pi:X \to Y$, with $X$ an irreducible SFT, 
such that $\pi^{-1}(y)$ contains a periodic point with the same least period as $y$.
\item There is an almost invertible cover $\pi:X \to Y$ 
%, with $X$ an irreducible SFT,  
such that $\pi^{-1}(y)$ contains a periodic point with the same least period as $y$.
\end{enumerate}
\end{cor}
%{\color{red}Remark: (2) has a self-contained %proof, and we have a similar proof for (3), but we %don't know how to simplify the proof.}

\begin{proof}
$(1)\Rightarrow (3)$: Let $\hat{\pi}: \hat{Y} \to Y$ be the factor code constructed in Proposition~\ref{AIcover_sofic} for a given periodic point $y$. Then 
$\hat{\pi}^{-1}(y)$ contains a synchronizing periodic point $\hat{y} \in \hat{Y}$ with the same least period as $y$.  Let $\gamma:X \to \hat{Y}$ be the (almost invertible) right Fischer covering map. {By Proposition~\ref{synchro_infinitelyoften}, $|\gamma^{-1}(\hat{y})|= 1$.} It follows that the unique preimage of $\hat{y}$ has the same least period as $y$.

{Let $\pi = \hat{\pi} \circ \gamma$. It is clear from the previous paragraph that $\pi^{-1}(y)$ contains a periodic point with the same least period as $y$. We now claim that $\pi$ is an almost invertible factor code, i.e., every doubly transitive point has exactly one preimage under $\pi$. To see this, first recall from Proposition \ref{AIcover_sofic} that $\hat{\pi}$ is almost invertible, and thus every doubly transitive point in $Y$ has a unique $\hat{\pi}$-preimage. Then, by our construction in Proposition \ref{AIcover_sofic}, $\hat{\pi}$ is finite-to-one (in fact, as the reader can check, $\hat{\pi}$ is at most $4$-to-one. Indeed, while we don't need this, one can show that evevy almost invertbile factor code on an irreducible sofic shift is finite-to-one). By \cite[Lemma 9.1.13]{LM}, the unique $\hat{\pi}$-preimage of every doubly transitive point in $Y$ is also doubly transitive in $\hat{Y}$. Since $\gamma$ is almost invertible, every doubly transitive point in $\hat{Y}$ has a unique $\gamma$-preimage in $X$. Thus, every doubly transitive point in $Y$ has a unique preimage via $\pi=\hat{\pi}\circ \gamma$. That is, $\pi$ is almost invertible.} 
%Then $\pi^{-1}(y)$ contains a point with the same least period as $y$, as desired. 

$(3)\Rightarrow (2)$ is trivial, and $(1) \Leftrightarrow (2)$ is the content of Theorem \ref{16-03-23a}. 
\end{proof}

Using Proposition \ref{AIcover_sofic}, we prove the following result which says  in particular  that {existence of an}
$S$-factorizable embedding and {existence of an} {\em $AI$-factorizable embedding} (i.e., a factorizable embedding in the sense of Definition \ref{factorizable_def} such that $\pi$ is almost invertible) are equivalent.
%, and we give necessary and sufficient conditions for the existence of such embeddings. 

\begin{thm}
Let $Y$ be an irreducible sofic shift. Let $Z$ be a subshift with $h(Z)<h(Y)$. Set $p:= \p(Y)$. Then the following are equivalent:
\begin{enumerate}
\item[(1)] $Z$ is $p$-periodic and $q_{np}(Z)\leq \rec_{np}(Y)$ for all $n$.
\item[(2)] There is an $S$-factorizable embedding of $Z$ into $Y$.
\item[(3)] There is an $AI$-factorizable embedding of $Z$ into $Y$.
\end{enumerate}	
\end{thm}	
\begin{proof}
(3)$\Rightarrow$ (2) is trivial.

The equivalence between (1) and (2) is just Lemma \ref{17-06-24b}. 
%\comm{Klaus}{Lemma 4.16!}

It remains to show (1) $\Rightarrow$ (3).  To prove this, we will construct an almost invertible cover $\pi: X\to Y$
%code $\pi: X\to Y$, with $X$ an irreducible SFT. 
such that $q_{np}(Z)\leq r_{np}(\pi)$ for all $n$, and then apply Theorem \ref{Sophie2}.

First recall from Proposition~\ref{growth} that 
$$
\lim_{n\to\infty} \frac{1}{np} \log s_{np}(Y)=h(Y).
$$
Since $h(Z)<h(Y)$ and $\limsup_{n\to \infty}\frac{1}{n} \log q_{n}(Z) \le h(Z)$, there exists a positive integer $N$ such that 
\begin{equation} \label{nlarge}
	q_{np}(Z)\leq s_{np}(Y) \qquad  \mbox{for all $n\geq N$}. 
\end{equation}
Let $\{\xi^{(1)}, \xi^{(2)},\cdots, \xi^{(M)}\}$ be the set of receptive periodic points of least period $ < Np$ in $Y$, where $M$ is the total number of such points.
 Applying Proposition \ref{AIcover_sofic} at most $M$ times, with $\xi$ replaced by $\xi^{(i)}$ at each step, we obtain an irreducible  sofic shift $\hat{Y}$, an almost invertible factor code $\hat{\pi}: \hat{Y}\to Y$, and a set of synchronizing periodic points $\{\hat{\xi}^{(1)}, \hat{\xi}^{(2)},\cdots, \hat{\xi}^{(M)}\}$ such that for each $1\leq i\leq M$, $\hat{\pi}(\hat{\xi}^{(i)})=\xi^{(i)}$ and $\per(\hat{\xi}^{(i)})=\per(\xi^{(i)})$. 
 
 %Thus $$q_{np}(Z) \le r_{np}(Y) \le |S_{np}(\hat{Y}) \cap \pi^{-1}(R_{np}(\hat{\pi}))|  ~~~ for~ all~ n < N, $$
 
% By item 2 of Proposition \ref{AIcover_sofic},  for all $n$, $s_{np}(Y) \le |S_{np}(\hat{Y}) \cap \pi^{-1}(R_{np}(\hat{\pi}))|$. Combining this with (\ref{nlarge}), we obtain that
 
 %$$q_{np}(Z) \le   |S_{np}(\hat{Y}) \cap \pi^{-1}(R_{np}(\hat{\pi}))|  ~~~~ for ~ all ~ n \ge N  $$.  Thus, 
 %\begin{equation} \label{alln}
%	q_{np}(Z)\leq |S_{np}(\hat{Y}) \cap \pi^{-1}(R_{np}(\hat{\pi}))| \qquad  \mbox{for all $n$}. 
%\end{equation}

  Let $\gamma: X\to \hat{Y}$ be the right Fischer cover of $\hat{Y}$ and define $\pi:= \hat{\pi}\circ \gamma$. Then $\pi$ is an almost invertible cover of $Y$ since both $\gamma$ and $\hat{\pi}$ are almost invertible. 
  %\comm{Klaus}{I'm still short of an argument for this.}
  %\comm{Brian}{The argument that I gave in my comment on p. 38 works for this. We want to show that every doubly transitive point in Y has a unique $\pi$-preimage in X. But we know, as you pointed out, that by the sofic version of Lemma 9.1.13 in [LM] that the preimages of doubly transitive points in Y are doubly transitive in $\hat{Y}$. And for the Fischer cover, each doubly transitive point has a unique preimage. Thus, each doubly transitive point in Y has a unique $\pi$-preimage.}

%\begin{equation} 
%	q_{np}(Z)\leq r_{np}(\tau) \qquad  \mbox{for all $n$}. 
%\end{equation}

 For all $n\geq N$, by item 2 of Proposition \ref{AIcover_sofic}, any $\eta\in S_{np}(Y)$ has a unique synchronizing $\hat{\pi}$-preimage (call it $\hat{\eta}$) of least period $np$ in $\hat{Y}$; 
and {by Proposition \ref{synchro_infinitelyoften}, $\hat{\eta}$ has a unique $\gamma$-preimage}, which has the same least period in $X$. Thus, $s_{np}(Y)\leq r_{np}(\pi)$,
%, {\color{blue} where $r_{np}(\pi)$ is the number of periodic points in $\hat{Y}$ that has a $\pi$-preimage with the same least period}. 
which, combined with (\ref{nlarge}), gives $q_{np}(Z)\leq r_{np}(\pi)$ for all $n\geq N$.
 
 To show $q_{np}(Z)\leq r_{np}(\pi)$ for all $n< N$, we consider $\{\xi^{(1)}, \cdots, \xi^{(M)}\}$. Recall that for each $1\leq i\leq M$, $\xi^{(i)}$ has a $\hat{\pi}$-preimage $\hat{\xi}^{(i)}$ with the same least period. Since $\hat{\xi}^{(i)}$ is a synchronizing periodic point, it has a unique $\gamma$-preimage, which must have the same least period as $\hat{\xi}^{(i)}$ (and therefore the same least period as $\xi^{(i)}$). Thus, $\rec_{np}(Y)\leq r_{np}(\pi)$ for all $n<N$ and therefore $q_{np}(Z)\leq r_{np}(\pi)$ for $n<N$. 

Finally, since $Z$ is $p$-periodic, $h(Z)<h(Y)$ and $q_{np}(Z)\leq r_{np}(\pi)$ for all $n$, by Theorem \ref{Sophie2}, there is a factorizable embedding of $Z$ into $Y$ which factors through $\pi$, an almost invertible factor code, as desired. 
\end{proof}

{\em Acknowledgement:} We are happy to acknowledge that Sophie MacDonald's theorem, the main result of ~\cite{M}, was the main inspiration for our work.

\end{document}